\documentclass[reqno,a4paper,12pt]{amsart}
\usepackage{latexsym, amsmath, amssymb, amsthm, a4}
\usepackage{graphicx}
\font\sc=rsfs10 at 12pt
\usepackage{amsfonts, bbold, bbm, marvosym}
\usepackage{amssymb}
\usepackage[mathscr]{eucal}
\usepackage{mathrsfs}

\newtheorem{theorem}{Theorem}[section]

\theoremstyle{remark}
\newtheorem{remark}[theorem]{Remark}

\renewcommand{\a}{\alpha}
\renewcommand{\b}{\beta}

\newcommand{\g}{\gamma}
\newcommand{\G}{\Gamma}
\renewcommand{\d}{\delta}

\newcommand{\y}{\eta}

\renewcommand{\i}{\iota}

\newcommand{\vk}{\varkappa}
\renewcommand{\l}{\lambda}
\renewcommand{\L}{\Lambda}
\newcommand{\m}{\mu}
\newcommand{\n}{\nu}
\newcommand{\x}{\xi}

\newcommand{\s}{\sigma}

\newcommand{\vs}{\varsigma}
\renewcommand{\t}{\tau}

\newcommand{\f}{\phi}

\newcommand{\F}{\Phi}

\renewcommand{\o}{\omega}
\renewcommand{\O}{\Omega}
\newcommand{\xbc}{\xb^{\circ}}

\newcommand{\C}{{\mathbb C}}
\newcommand{\R}{{\mathbb R}}

\newcommand{\nup}{\pmb{\pmb{\n}}}

\newcommand{\bx}{\pmb{\xi}}

\newcommand{\db}{{\mathbf d}}
\newcommand{\eb}{{\mathbf e}}

\newcommand{\kb}{{\mathbf k}}

\newcommand{\pb}{{\mathbf p}}

\newcommand{\rb}{{\mathbf r}}

\newcommand{\ub}{{\mathbf u}}

\newcommand{\xb}{{\mathbf x}}
\newcommand{\xt}{{\mathbf x}^{\bullet}}
\newcommand{\yb}{{\mathbf y}}
\newcommand{\zb}{{\mathbf z}}

\newcommand{\mm}{{\mathbbm{m}}}

\newcommand{\Eb}{{\mathbf E}}
\newcommand{\Fb}{{\mathbf F}}
\newcommand{\Gb}{{\mathbf G}}

\newcommand{\Ib}{{\mathbf I}}

\newcommand{\Nb}{{\mathbf N}}

\newcommand{\Qb}{{\mathbf Q}}

\newcommand{\Tb}{{\mathbf T}}

\newcommand{\kn}{\mathbbm{k}}

\newcommand{\km}{\mathbbm{k}}
\newcommand{\aF}{\mathfrak a}
\newcommand{\AF}{\mathfrak A}
\newcommand{\bF}{\mathfrak b}
\newcommand{\BF}{\mathfrak B}

\newcommand{\gF}{\mathfrak g}

\newcommand{\hF}{\mathfrak h}

\newcommand{\jF}{\mathfrak j}

\newcommand{\kF}{\mathfrak k}
\newcommand{\KF}{\mathfrak K}

\newcommand{\mF}{\mathfrak m}
\newcommand{\MF}{\mathfrak M}

\newcommand{\qF}{\mathfrak q}
\newcommand{\QF}{\mathfrak Q}

\newcommand{\rF}{\mathfrak r}
\newcommand{\RF}{\mathfrak R}
\newcommand{\ssF}{\mathfrak s}
\newcommand{\SF}{\mathfrak S}

\newcommand{\uF}{\mathfrak u}

\newcommand{\ZF}{\mathfrak Z}
\newcommand{\zF}{\mathfrak z}

\newcommand{\vF}{\mathfrak v}

\newcommand{\Ac}{{\mathcal A}}
\newcommand{\Bc}{{\mathcal B}}

\newcommand{\Dc}{{\mathcal D}}

\newcommand{\Fc}{{\mathcal F}}

\newcommand{\Kc}{{\mathcal K}}
\newcommand{\Lc}{{\mathcal L}}

\newcommand{\Qc}{{\mathcal Q}}
\newcommand{\Rc}{{\mathcal R}}
\newcommand{\Sc}{{\mathcal S}}

\newcommand{\As}{\sc\mbox{A}\hspace{1.0pt}}

\newcommand{\Fs}{\sc\mbox{F}\hspace{1.0pt}}

\newcommand{\Bs}{\sc\mbox{B}\hspace{1.0pt}}

\newcommand{\dist}{{\rm dist}\,}
\newcommand{\diag}{{\rm diag}\,}

\newcommand{\grad}{{\rm grad}\,}

\renewcommand{\div}{{\rm div}\,}

\newcommand{\pd}{\partial} 

\newcommand{\diam}{\operatorname{diam\,}}

\newcommand{\Tr}{\operatorname{Tr\,}}
\newcommand{\tr}{\operatorname{tr\,}}

\newcommand{\pnu}{{\partial_{\nup}}}

\newcommand{\bL}{\pmb{\L}}

\theoremstyle{plain}

\newtheorem{proposition}[theorem]{Proposition}
\newtheorem{corollary}[theorem]{Corollary}
\newtheorem*{theorem*}{Theorem}

\theoremstyle{definition}

\renewcommand{\div}{\mbox{div}\,}

\newcommand{\la}{\langle}
\newcommand{\ra}{\rangle}

\newcommand{\beq}{\begin{equation}}
\newcommand{\eeq}{\end{equation}}

\numberwithin{equation}{section}
\numberwithin{figure}{section}

\begin{document}

\title[Neumann-Poincar\'e in 3D elasticity]{The discrete spectrum of the Neumann-Poincar\'e operator in 3D elasticity.}

\author{Grigori Rozenblum} 


\email{grigori@chalmers.se}


\begin{abstract}
For the Neumann-Poincar\'e (double layer potential) operator in the three-dimensional  elasticity we establish asymptotic formulas for eigenvalues converging to the points of the essential spectrum and discuss geometric and mechanical meaning of coefficients in these formulas. In particular, we establish that for any body, there are infinitely many eigenvalues converging from above to each point of the essential spectrum. On the other hand, if there is a point where the boundary is concave (in particular, if the body contains cavities) then for each point of the essential spectrum there exists a sequence of eigenvalues converging to this point from below. The reasoning is based upon the representation of the Neumann-Poincare operator as a zero order  pseudodifferential operator on the boundary and the earlier results by the author on the eigenvalue asymptotics for polynomially compact pseudodifferential operators.

\end{abstract}
\maketitle




\let\thefootnote\relax\footnote{ Chalmers University of Technology(Sweden); St.Petersburg State University, The Euler International Mathematical Institute (St.Petersburg), Sirius University (Sochi, Russia),\\ email: grigori@chalmers.se}

\section{Introduction}\label{Intro} The paper is devoted to the study of the spectrum of the Neumann-Poincar\'e (NP) operator in the 3D linear elasticity. It is based upon  results of the previous paper \cite{RozNP1}, where we considered general polynomially compact pseudodifferential operators and have derived formulas describing for such operators the behavior of eigenvalues converging to the points of the essential spectrum. The motivating example, the Neumann-Poincar\'e (the double layer potential) operator $\KF$ in 3D elasticity was presented, and a discussion of spectral properties of this operator has started.   The present paper continues the study of the eigenvalues of the Neumann-Poincar\'e elasticity operator for a homogeneous and isotropic 3D body $\Dc$ with smooth boundary $\G$ on the base of results in \cite{RozNP1}. It is known, since \cite{3D}, \cite{MR3D}, that this operator possesses three points of essential spectrum, namely, the zero point  and two symmetrical ones, $\pm \mathbbm{k}$, where $\mathbbm{k}$ is expressed via the Lam\'e constants of the material of the body and does not depend on its geometry. There may also exist finite or infinite  sequences of eigenvalues, converging (in the latter case) to the points of the essential spectrum. In this paper we find sufficient geometrical conditions for these sequences to be infinite (above or below of a point of the essential spectrum), and if this is the case, we study  asymptotic properties  of these sequences, following the general results in \cite{RozNP1}.

In fact, up to now, very little was known about the discrete spectrum of the operator $\KF$ beyond the case of the sphere (where the spectrum has been, rather recently, calculated explicitly, see \cite{DLL}).  In the general case, some estimates for the rate of convergence of these eigenvalues have been found in \cite{3D}.

For the case of the ball with radius $R$ and Lam\'e constants $\l,\m$, the eigenvalues of the NP operator, calculated in \cite{DLL}, form three series,

  \begin{gather}\label{ball}
\L_n^{0}(\KF)=\frac{3}{2(2n+1)}\sim \frac{3}{4n},\, n\to\infty \\\nonumber
\L_n^{-}(\KF)=\frac{3\lambda-2\mu(2n^2-2n-3)}{2(\lambda+2\mu)(4n^2-1)}\sim-\mathbbm{k}+ \frac{4\km}{n} ,\,n\to\infty \\\nonumber
\L_n^{+}(\KF)=\frac{-3\lambda+2\mu(2n^2+2n-3)}{2(\lambda+2\mu)(4n^2-1)},\sim\mathbbm{k}+ \frac{4\km}{n},\, n\to\infty,
\end{gather}
 $\mathbbm{k}=\frac{\m}{2(2\m+\l)},$
 each of $\L_n^{-}(\KF),$ $ \L_n^{0}(\KF), $ $\L_n^{+}(\KF)$ being a multiple eigenvalue with multiplicity $2n+1$. Additionally, it is found that no eigenvalues coincide with the points of the essential spectrum.
One can see that the eigenvalues $\L_n^{\pm}(\KF)$ (as well as even their  asymptotics) depend on the material of the body. A possible dependence on the geometry is concealed here,  due to the fact that the spectrum of $\KF$ is invariant under the homotheties of the body, so the only geometric characteristic of the sphere, its radius, is not present in the formulas for eigenvalues.   One can also notice that all three series converge to their limit values from above only; there are no eigenvalues that approach these points from below. This latter property is important, in particular, in the analysis of the plasmon resonance in elastic metamaterials, see, e.g., \cite{AndoReson}, \cite{Ammari20}, \cite{Li}, \cite{AndoEll}, \cite{AKM}, \cite{Ando2D},  \cite{LiLiu}, therefore, it is interesting to determine, to what extent these properties persist in a more general case.

Recently, the spectral problem has been studied for the NP operator in electrostatics, where the operator is compact. There, in dimension 3, for a smooth boundary, the eigenvalue asymptotics, a power-like one, was found in \cite{M}, \cite{MR}, with somewhat weaker results for the case of a finite smoothness. In dimension 2, the rate of convergence of eigenvalues to zero depends on the smoothness of the boundary; it is at least polynomial for a finite smoothness, super-polynomial for an infinitely smooth boundary, and (again, at least) exponential for an analytic boundary. Only upper estimates for eigenvalues are known. In the only case where the eigenvalue asymptotics is found, namely, for an ellipse,  the asymptotics is exponential (see \cite{M}).

 The situation is similar for the  elastic NP operator in dimension 2, where it has two points of the essential spectrum (\cite{AKM}, \cite{Ando2D}, \cite{Ammari20}); it was found that the rate of convergence of eigenvalues to these points depends again on the smoothness of the boundary.  Namely, the estimates obtained in these papers show that for a finitely smooth boundary, the eigenvalues converge to their limit points at least polynomially fast,  for an infinitely smooth boundary these eigenvalues converge super-polynomially fast, while for an analytic boundary they converge at least exponentially. An exact asymptotics of eigenvalues was never found, even for an ellipse.

In the present paper we consider a  body $\Dc\subset R^3$ made of a homogeneous isotropic elastic material with Lam\'e constants $\l,\m$ and  bounded by a smooth compact  surface $\G$. The NP operator is a polynomially compact zero order pseudodifferential operator, with three points $\o_{\i}: \o_{-1}=-\km, \o_0=0, \o_1=\km,$ in the essential spectrum, according to the results of \cite{AgrLame}, \cite{MR3D}. The asymptotics of eigenvalues of such operator tending to a point $\o_\i,$ $\i=-1,0,1,$ of the essential spectrum, as found in \cite{RozNP1}, is determined by a certain pseudodifferential operator $\MF_{\i}$ of order $-1$. The procedure for calculating the principal symbol $\mF_\i$ of $\MF_{\i}$ is quite intricate. Moreover, the calculation of the coefficients in the eigenvalue asymptotic formulas involves the eigenvalues of the latter symbol, a $3\times 3$ symbolic matrix; it presents the so-called irreducible case of the cubic equation. Therefore the symbolic expression for integrals of powers of these eigenvalues in the general case would be completely unwieldy and, even if found, be of no use for further analysis.

 In our approach, using the qualitative analysis of the NP operator as a singular integral operator,  we are able first to separate the dependence of the principal symbol $\mF_\i$ on the geometric characteristics of the surface and the dependence on the material of the body. This development is achieved thanks to determining the qualitative structure  of the symbol $\mF_\i$ as described in \cite{RozNP1}. Namely, we establish that, although the expression of the symbol $\mF_\i(x,\x),\,(x,\x)\in \mathrm{T}^*(\G),$ contains 25 additive terms, each being the product of  5 symbolic matrices, it is one and only  one factor in each such product that depends on the geometry of $\G,$ namely, it is a linear combination of principal curvatures of $\G$ at the point $x\in\G,$ with (matrix) coefficients depending universally on $\x$ and on the Lam\'e constants.  This leads to our structure result: the combination of these terms, the symbol $\mF_\i(x,\x)$, is a linear form of the principal curvatures with universal (depending only on the Lam\'e constants and $\x$) coefficients.

 Then a miraculous circumstance helps us. The Birman-Solomyak formula (\cite{BS}) for coefficients in the eigenvalue asymptotics for a negative order pseudodifferential operator  involves the trace of a certain, generally fractional, power of the  principal symbol or of its positive or negative parts, in other words, the sum of powers of (all or a part of)  eigenvalues of the matrix symbol $\mF_\i(x,\x)$. Calculating this sum requires, generally, knowing the eigenvalues themselves, and this, for $3\times 3$ matrices, cannot be achieved in the symbolic way, as was just discussed. The situation seems hopeless. However, there are exceptions. If the power is \emph{an  integer}, the sum of powers of all eigenvalues of a Hermitian matrix (or a matrix similar to a Hermitian one, as it happens in our case) can be found without knowing the eigenvalues  themselves, rather only by using polynomial operations with entries of the matrix. And, luckily, for the case of the NP operator on a two-dimensional surface, this power equals exactly $2.$  Therefore, the integrand in the Birman-Solomyak formula is the sum of squares of linear forms in principal curvatures, therefore, it is a \emph{quadratic form} of these curvatures, again, with universal coefficients depending on the Lam\'e constants only.

 This way of reasoning   enables us to determine the two-sided asymptotics for the eigenvalues of the NP operator, to say it more exactly, the asymptotics for the sum of the counting functions for eigenvalues above and below $\o_\i$. The reasoning explained above leads also to the fact that the geometrical characteristics entering in these formulas are the Euler characteristic of the surface $\G$ and its Willmore energy $W(\G)$ (see, e.g., \cite{Willmore} for a discussion of classical and recent problems and results concerning this latter quantity which indicates  how the surface in $\R^3$ is bent.)
 Such kind of formulas is  similar to the case of the NP operator in 3D electrostatics, derived in \cite{M}, \cite{MR}.

An important question concerns the infiniteness of the sequences of eigenvalues tending to the points of the essential spectrum separately from above and from below. The pattern obtained for the symbol  $\mF_\i(x,\x)$, symmetry considerations, and the  above asymptotic formulas show  that
 the sequences of eigenvalues converging to $\o_\i$ \emph{from above} are always infinite.  We also  find a sufficient condition in geometric terms for the \emph{infiniteness} of the sequences of eigenvalues converging to $\o_\i$ \emph{from below:} this happens, in particular, for sure, if there is at least one  point on $\G$ where the surface is concave or, more generally, where the mean curvature in a special co-ordinates system is positive. It deserves noticing that for a body with a cavity, where $\G$ is not connected, this happens always.

In the last section a more detailed analysis which shows in what way the coefficients in the effective  symbol $\mF_{\i}$ should depend on the material characteristics, the Lam\'e constants. This task requires a more detailed analysis of the process of calculation of the symbol and subsymbol of the NP operator.  The universality properties in the structure of the symbol $\mF_\i$ enable us to reduce this calculation to a unique particular case, namely of the surface being a cylinder with circular cross-section, where one of the  principal curvatures vanishes. The treatment of this case still requires a considerable calculational work, but it is at least more feasible as long, at least, as it concerns qualitative properties. As a result of our calculations, we establish that this symbol is a linear form of the quantities $\km=\frac{\m}{2(2\m+\l)}$ and $\mm=\frac12=\km$ with coefficients, now, depending only on $\i$ and on $\x\in\S^1.$ Some further properties of these coefficients are derived, using the universality, from the result for the case of the sphere.

  We should mention that an alternative approach to the elastic  NP spectral problem can, probably, be based upon  recent results on diagonalizing   matrix pseudodifferential  operators, see \cite{Capo.Diag}, \cite{CapoVas.Diag}; the important initial step, the global diagonalization of the principal symbol, is possible according  to the results of \cite{CR.Diag}. We plan to explore this approach in the future.

  The Author is appreciative
 to Y.Miyanishi for introducing him to the NP problematic as well as for useful discussions.

\section{Preliminaries}\label{results}
Let $\Dc\subset\R^3$ be a bounded (connected) body with smooth boundary $\G.$ It may happen that the surface $\G$ is not connected, namely,  in the case when the body has some cavities.

We write the Lam\'e system for a homogeneous isotropic body  $\Dc$ in the form
 \begin{gather*}
     \Lc \ub\equiv \Lc_{\l,\m}\ub\equiv -\m\div(\grad \ub)-(\l+\m)\grad(\div \ub)=0,\\  \nonumber
      \xb=(x_1,x_2,x_3)\in\Dc, \ub=(u_1,u_2,u_3)^\top,
 \end{gather*}
 where  $\l,\m$  are the Lam\'e constants.
 The fundamental solution  $\Rc(\xb,\yb)=[\Rc(\xb,\yb)]_{p,q},$ ${p,q=1,2,3},$ for the Lam\'e equations, \emph{the Kelvin matrix},  known since long ago, see, e.g., \cite{KuprPot},
 equals
 \begin{gather*}
    [\Rc(\xb,\yb)]_{p,q}=\l'\frac{\d_{p,q}}{|\xb-\yb|}+\m'\frac{(x_p-y_p)( x_q-y_q)}{|\xb-\yb|^3},\\ \nonumber \l'=\frac{\l+3\m}{4\pi \m(\l+2\m)},
    \, \m'=\frac{\l+\m}{4\pi\m(\l+2\m)}, \,\, p,q=1,2,3,\,\, \xb,\yb\in\R^3.
 \end{gather*}
 This expression can be found, in particular, by inverting the Fourier transform of the symbol $\rb(\bx)$ of $\Lc^{-1}$:
 \begin{equation}\label{KelvinF}
    \Rc(\xb,\yb)=\Fc^{-1}[\rb](\xb-\yb)\equiv(2\pi)^{-3}\int_{\R^3} e^{i(\xb-\yb)\bx} (\m \bx\bx^\top+(\l+\m)|\bx|^2\Eb)^{-1} d\bx,
 \end{equation}
where $\bx$ is treated as a column-vector, so that $\bx\bx^\top$ is a  $3\times 3$ square matrix; $\Eb$ is the unit $3\times 3$ matrix.

The classical boundary problems for the Lam\'e system involve the coboundary (traction) operator
\begin{equation*}
    [\Tb(\xb,\partial_{\nup(\xb)})]_{p,q}=\l \n_p\pd_q+\l\n_q\pd_p+\m\d_{p,q}\pnu,\, p,q=1,2,3,
  \end{equation*}
  where $\nup=\nup(\xb)=(\n_1,\n_2,\n_3)$ is the  outward normal unit vector to $\G$ at the point $\xb$ and $\pnu_{(\xb)}$ is the
   directional derivative along $\nup(\xb)$.

 The NP operator $\KF$ on $\G$ is defined as
 \begin{equation}\label{DoubleLayer}
    (\KF[\psi])(\xb)=\int_{\G}\Kc(\xb,\yb)\psi(\yb)d\Sc(\yb)\equiv \int_{\G} \Tb(\yb,\pnu_{(\yb)}) \Rc(\xb,\yb)^{\top}\psi(\yb)d\Sc(\yb), \xb\in \G,
 \end{equation}
 where $d\Sc$ is the natural surface  measure  on $\G,$ the Riemannian measure generated by the embedding of $\G$ in $\R^3,$ and $\Tb(\yb,\pnu_{(\yb)})$ denotes the coboundary operator at the point $\yb\in\G.$ The explicit expression for the kernel $\Kc(\xb,\yb)$ is known since long ago as well; it is given, e.g., in \cite{Kupr79}, Sect.4, Ch.2:
 \begin{gather}\label{kernelNP}
     [\Kc(\xb,\yb)]_{p,q}=\frac{\mathbbm{k}}{2\pi}\frac{\n_p(\yb)(x_q-y_q)-\n_q(\yb)(x_p-y_p)}{|\xb-\yb|^3}-\\ \nonumber
    -\frac1{2\pi}\left(\km \d_{p,q}+ 3\mm \frac{(x_p-y_p)(x_q-y_q)}{|\xb-\yb|^2}\right)\sum_{l=1}^3 \n_l(\yb)\frac{x_l-y_l}{|\xb-\yb|^3};\, p,q=1,2,3,
 \end{gather}
 where
$$\mm=\frac{\l+\m}{2(\l+2\m)}=\frac12-\kn.$$

 Since the boundary is smooth,  $\KF$ is a pseudodifferential operator of order zero on the surface $\G$, i.e., a singular integral operator on  $\G$; the leading singularity at the diagonal $\yb=\xb\in \G$ of the  kernel, determined by off-diagonal terms in \eqref{kernelNP}, is of order $-2$ and it is odd in $\yb-\xb$ as $\yb\to\xb$. To express the symbol of $\KF$ as a pseudodifferential operator, a local co-ordinate system on $\G$ and a frame in $\R^3$ are fixed. Following \cite{AgrLame}, for a fixed point $\xb^\circ\in \G$, two co-ordinate axes $x=(x_1,x_2)$ are orthogonally  directed in the tangent plane to $\G$ at $\xb^\circ$, and the third axis $x_3$ is directed orthogonally, in the \emph{outward} direction, so that the surface $\G$ near $\xbc$ is described by the equation $x_3=F(x_1,x_2)$, and in these co-ordinates, $F(0,0)=0,\, \nabla F(0,0)=0$. These $x_1,x_2$ are chosen as local co-ordinates near $\xb^\circ$ on $\G$. The dual co-ordinates $\x=(\x_1,\x_2)$ in the tangent plane are directed along the same spacial axes in $\texttt{T}(\G)$. The corresponding  vectors will be accepted as the basis in the cotangent plane at $\xb^\circ$ (identified naturally with the tangent plane.) The same vectors as well as the normal $\nup(\xb^{\circ})$ are accepted as the frame in the fiber $\C^3$ over $\G$ near $\xb^\circ.$ In these co-ordinates, the principal symbol of $\KF$ is calculated in  \cite{AgrLame} to be equal to
 \begin{equation}\label{NPPrinc}
    \kF_0({x},\x)=\frac{i \mathbbm{k}}{|\x|}\begin{pmatrix}
                    0 & 0 & -\x_1 \\
                    0 & 0 & -\x_2 \\
                    \x_1 & \x_2 & 0 \\
                  \end{pmatrix},
 \end{equation}
 where $\mathbbm{k}=\frac{\m}{2(2\m+\l)}$, and the eigenvalues of this matrix,  $\o_{-1}=-\mathbbm{k}, \o_0=0$, $\o_1=\mathbbm{k}$, are the points of the essential spectrum of $\KF$. Note that these eigenvalues do not depend on the geometry of the body $\Dc.$ At this point $\xb^\circ,$ $dx_1dx_2$ equals the area element for the surface measure on $\G$ generated from the Lebesgue measure in $\R^3$
 by the embedding  $\G\subset\R^3.$.

 The eigenvectors of the principal symbol \eqref{NPPrinc} equal $\eb_{\pm}=2^{-\frac12}|\x|^{-1}(\x_1,\x_2,\pm \imath|\x|)^\top$ for the eigenvalues $\pm\kn$ and $\eb_0=|\x|^{-1}(\x_1,-\x_2,0)$ for the eigenvalue $0.$ In the literature cited above, one can encounter some considerations concerning the mechanical meaning of such form of eigenvectors. In particular, there is an assumption that the eigenfunctions of the NP operator, corresponding to its eigenvalues close to $0$, describe surface waves which are 'almost purely' compression ones, while two other eigenvectors describe 'almost incompressible' waves lying close to the subspace $\mathrm{div}\ub=0$. We do not know to what extent our results support these speculations. An essential progress on this topic seems to have been  made in a very recent preprint \cite{Fukushima}. There, a splitting of the space of vector-functions on $\G$ into three subspaces has been constructed, so that, on the one hand, the subspaces approximate spectral subspaces of $\KF$ corresponding to its spectrum around the points $\o_\i,$ and, on the other hand, functions in these subspaces possess extensions inside or outside $\G$ to functions with special properties, divergence and/or rotor-free.

 Due to the results of \cite{RozNP1}, in order to find the asymptotics of eigenvalues of $\KF$, we   need also expressions for the subsymbol of the operator $\KF$ and the derivatives of the principal symbol. These objects depend
 essentially on the choice of local co-ordinates and frame. The convenient choice, refining the one described above, will be determined later on, enabling considerable simplification in our calculations. We use the notion of 'subsymbol' for the symbol of order $-1$ of a zero order  pseudodifferential operator in a fixed co-ordinate system and a fixed frame. Unlike the classical notion of a 'subprincipal symbol' which is invariant under the change of local co-ordinates, the 'subsymbol' is not invariant, but the results of our eigenvalue calculations with the subsymbol turn out to be invariant.

 As  established in \cite{RozNP1}, the asymptotic behavior of the eigenvalues of  $\KF$, as they approach the points $\o_{\i}$,  is determined by the eigenvalue behavior for the compact operators $\MF_{\i}=\pb_{\i}(\KF)$, where $\pb_{\i}(\o)$ are polynomials specially constructed according to the eigenvalues of the principal symbol of the operator, see Lemma 4.1 in \cite{RozNP1}. In our case, the dimension $\Nb$ of the vector bundle, where the operator $\KF$ acts, equals 3 and all eigenvalues of the principal symbol are simple. Therefore, the number $\i$ takes values $-1,0,1,$ and the degree of the polynomial $\pb_\i(\o)$ equals 5. By Lemma 4.1 in \cite{RozNP1}, this means that the polynomial $\pb_\i(\o)$
 has the form
  \begin{equation}\label{p(o)}
    \pb_{\i}(\o)=(\o-\o_\i)\prod_{\i'\ne\i}(\o-\o_{\i'})^2, \,\i=-1,0,1.
 \end{equation}

 We are interested in the principal, order $-1$, symbol $\mF_{\i}=\mF_{\i,-1}(x,\x)$ of the operator $\MF_{\i}$; we call it the effective symbol. In \cite{RozNP1}, Proposition 4.2, the structure of this symbol was described.
Due to this Proposition, $\mF_{\i}$ is the sum of terms of 2 types. In order to write down them in a systematic way, we consider the sets  $\mathbb{J}_{\i} =\{\vs_j\}_{j=1,\dots,5},$ $\i=-1,0,1;$  the set $\mathbb{J}_{\i}$ consists of the integers $-1,0,1$, in such way that each of them is repeated twice, except $\i$ which is repeated only once, placed in the nondecreasing order, thus, $\mathbb{J}_{-1}=\{-1,0,0,1,1\}$, $\mathbb{J}_0=\{-1,-1,0,1,1\}$, $\mathbb{J}_1=\{-1,-1,0,0,1\}$. Then, by Proposition 4.2 in \cite{RozNP1}, the terms in $\mF_\i$ of type 1 have the form
 \begin{equation}\label{type1}
 \Fb_{l,\i}=\prod_{ j< l}(\kF_0-\o_{\vs_j})\kF_{-1}\prod_{j>l}(\kF_0-\o_{\vs_j}),\, \vs_j\in\mathbb{J}_{\i},\,\i=-1,0,1,\,\ l=1,\dots,5,\,
 \end{equation}
where $\kF_{-1}$ is the subsymbol  of the operator $\KF$ --  altogether 5 terms for each $\i$.
 The terms of type 2 have the form
\begin{gather}\label{type2}
\Gb_{\a,l,m,\i}=\frac1i\prod_{j<l}(\kF_0-\o_{\vs_j})(\partial_{\x_\a}\kF_0)
\prod_{l<j<m}(\kF_0-\o_{\vs_j})(\partial_{x_\a}\kF_0)\prod_{j>m}(\kF_0-\o_{\vs_j}),\\\nonumber \, \a=1,2; \, 1\le l< m\le 5,
\end{gather}
altogether 20 terms for each $\i$.

One should keep in mind that all factors in \eqref{type1}, \eqref{type2} are $3\times 3$ matrices. The factors of the form $\kF_0-\o_{\vs_j}$ commute with each other but the ones containing the  derivatives of $\kF_0$ and the subsymbol do not commute with  $\kF_0-\o_{\vs_j}.$ One can also see in  \eqref{type1}, \eqref{type2} that the  terms of type 1 contain $4$ factors  $\kF_0-\o_{\vs_j}$ while terms of type 2 contain $3$ such factors. Finally, the symbol $\mF_{\i}$ is equal to
\begin{equation}\label{complete symbol}
    \mF_{\i}=[\pb'_\i(\o_\i)]^{-1}\left(\sum_{l=1}^{5} \Fb_{l,\i}+\sum_{\a=1,2}\sum_{1\le l< m\le 5}\Gb_{\a,l,m,\i}\right),
\end{equation}
25 terms altogether, for each $\i=-1,0,1$.

In Sect.6, we discuss some economy approach for calculating the expression \eqref{complete symbol}. This gives us a possibility to describe the dependence of the effective symbol on the Lam\'e constants.

As was emphasized  earlier, the expressions in \eqref{type1}, \eqref{type2} contain factors which are not invariant with respect to changes of local co-ordinates on the surface $\G$ and the frame. This is unavoidable but not important, since the \emph{sum} of all such terms in \eqref{complete symbol}, being the \emph{principal}, order $-1$, symbol $\mF_{\i, -1}(x,\x)$ of the operator $\MF_{\i}$, is in the usual sense invariant with respect to  changes of local co-ordinates on $\G$ and may depend only on the frame; its eigenvalues are  invariant. We will use this invariance essentially when choosing  the co-ordinates and the frame in a convenient way. Moreover,  the operator $\pb_{\i}(\KF)$ and its principal symbol $\mF_{\i,-1}$ as well do not depend on the order in which the terms $\KF-\o_{\vs_j}$ are multiplied, although each particular term in   \eqref{complete symbol} does.

A special feature of the operator $\KF$ is the fact that it is not self-adjoint in $L^2(\G)$ (with the surface measure generated by the Lebesgue measure in $\R^3$). This operator is, however, \emph{symmetrizable}. This property was discussed in \cite{RozNP1}, Section 6.  Recall that the pseudodifferential operator $\KF$ is called symmetrizable if there exists a positive elliptic pseudodifferential operator $\SF$ such that
\begin{equation*}
    \AF=\SF^{-1/2}\KF\SF^{1/2},
\end{equation*}
is self-adjoint in $L^2(\G)$ or, equivalently,
\begin{equation}\label{Symmetr2}
    \SF\KF^*=\KF\SF.
\end{equation}
In Sect. \ref{Sect.symm} we discuss the symmetrizability of our  operator $\KF$ in more detail.

\section{Symmetrization and general asymptotic formulas}\label{Sect.symm}
\subsection{Symetrization}
   The fact that  the operator $\KF$ is not self-adjoint in $L^2(\G)$ is already easily visible from its definition \eqref{DoubleLayer} -- the adjoint operator involves the normal derivative at the point $\xb$, and not at the point $\yb$, as in \eqref{DoubleLayer}. This shortcoming can be circumvented by showing that $\KF$ is symmetrizable.

    Consider the \emph{single layer potential operator} on $\G$:
    \begin{equation}\label{SingleLayer}
        \SF[\psi](\xb)=\int_\G \Rc(\xb,\yb)\psi(\yb) d\Sc(\yb),\, \xb\in \G,
    \end{equation}
    the kernel $\Rc$ being defined in \eqref{KelvinF}.
    This is a self-adjoint operator in $L^2(\G)$. It is well known, see, e.g.  \cite{AgrLame}, that $\SF$
    (it is denoted by $A$ there) is an elliptic
    pseudodifferential operator of order $-1$. Therefore, $\SF$ maps the Sobolev space $H^s(\G)$
     into the space $H^{s+1}(\G)$ for any $s\in(-\infty,\infty)$.
     The principal symbol of $\SF$ has been calculated  in  \cite{AgrLame}, Sect. 1.6.
      In the local co-ordinates and the frame just used above, it has the block-matrix form
      \begin{equation*}\label{SingleSymbol}
        \ssF_{-1}(x,\x)=\frac{1}{2\m|\x|}\left(\Eb -\mm\begin{pmatrix}
                                             \bL(\x) & 0 \\
                                             0 & 1 \\
                                           \end{pmatrix}\right), \mm=\frac{\l+\m}{2(\l+2\m)}=\frac12-\km.
      \end{equation*}
Here $\bL(\x)$ is the  $2\times2$ matrix
\begin{equation*}
    \bL(\x)=|\x|^{-2}\begin{pmatrix}
         \x_1^2 & \x_1\x_2 \\
        \x_1\x_2& \x_2^2 \\
       \end{pmatrix},
\end{equation*}
$\Eb$ denotes the unit $3\times 3$ matrix.

Matrix \eqref{SingleSymbol} is invertible, therefore, the operator $\SF$ is elliptic. We need some more, namely, that $\SF$ is positive in $L^2(\G)$.

\begin{proposition}\label{Prop.Positive} The single layer potential $\SF$ in \eqref{SingleLayer} is a positive operator in $L^2(\G)$, $\la\SF \psi,\psi\ra_{L^2(\G)}>0.$
 \end{proposition}
 \begin{proof}
 In the scalar case, for the single layer electrostatic potential, this property is well-known,
 see, e.g.,  \cite{Landkof}, Theorem 1.15. We failed to find an exact reasoning for the elastic case in the literature, therefore we present an elementary proof here.
 Denote by $\Qc(\xb,\yb)$ the fundamental solution for the square root of the  Lam\'e operator $\Lc$. This function can be constructed as
 \begin{equation*}
    \Qc(\xb,\yb)=(2\pi)^{-3/2}\int_{\R^3} e^{i(\xb-\yb)\bx}\sqrt{\rb(\bx)}d\bx,
 \end{equation*}
 with $\sqrt{\cdot}$ denoting here the positive square root of a positive matrix. Since $\sqrt{\rb(\bx)}\times\sqrt{\rb(\bx)}=\rb(\bx)$, the kernel $\Qc$ satisfies
 \begin{equation}\label{double}
    \int_{\R^3} \Qc(\xb,\zb)\Qc(\zb,\yb) d\zb =\Rc(\xb,\yb), \xb\in \Dc.
 \end{equation}
 Using \eqref{double}, we can represent the single layer potential  operator as $\SF=\Qb^*\Qb$, where $\Qb$ is the operator acting from $L^2(\G)$ to $L^2(\R^3)$ as
\begin{equation*}
 \Qb[\psi](\xb)=\int_\G \Qc(\xb,\yb)\psi(\yb)d\Sc(\yb).
\end{equation*}
 This representation shows that the operator $\SF$ is nonnegative. Finally, in accordance with \cite{AgrLame}, Proposition 1.2, the null space of $\SF$ is trivial, so $\SF$ is positive.
 \end{proof}

 Taking into account the ellipticity of $\SF$, we know now that $\SF$ is an \emph{isomorphism} of Sobolev spaces, $\SF: H^s(\G)\to H^{s+1}(\G), \, -\infty<s<\infty$. Moreover, any power of $\SF$ is an isomorphism $\SF^\g: H^s(\G)\to H^{s+\g}(\G)$, $-\infty<\g<\infty.$

The matrix $\bL=\bL(\x)$ satisfies $\bL^2=\bL$, this property enables us to calculate explicitly principal  symbols of some operators related with $\SF$. First, the inverse $\RF=\SF^{-1}$ is a pseudodifferential operator of order 1. Its principal symbol $\rF_1$ equals $\ssF^{-1}$,
 \begin{equation*}
    \rF_1(\x)=2\m|\x|\left(\Eb +\frac{\l+\m}{\l+3\m}\begin{pmatrix}
                    \bL(\x) & 0 \\
                    0 & 1 \\
                  \end{pmatrix}\right)=(\ssF_{-1}(\x))^{-1}.
 \end{equation*}

  We will also need  the  (positive) square roots from the operators $\SF$ and $\RF$. The operator $\QF=\SF^{\frac12}$ is an elliptic pseudodifferential operator
 of order $-\frac12$ and its principal symbol equals
\begin{equation*}\label{qF}
    \qF_{-\frac12}(\x)=(\ssF_{-1}(\x))^{\frac12}=\frac{1}{(2\m|\x|)^{\frac12}}\left(\Eb-(1-(1-\mm)^{\frac12})
    \begin{pmatrix}
    \bL(\x) & 0 \\
     0 & 1 \\
     \end{pmatrix}\right).
     \end{equation*}

 In its turn, the principal symbol of the  order $\frac12$ pseudodifferential operator $\ZF=\QF^{-1}=\RF^{\frac12}$ equals
\begin{gather}\label{zF}
    \zF_{\frac12}(\x)=(\rF_1(\x))^{\frac12}=(\qF_{-\frac12}(\x))^{-1}=\\\nonumber
    (2\m|\x|)^{\frac12}\left(\Eb+ (\frac{1}{\sqrt{1-\mm}}-1)\begin{pmatrix}
                    \L(\x) & 0 \\
                    0 & 1 \\
                  \end{pmatrix}\right),
\end{gather}

We can show now that  our operator $\KF$ is symmetrizable in $L^2(\G)$, with  the single layer potential operator  $\SF$ acting as symmetrizer. In fact, the relation \eqref{Symmetr2} has been known since long ago; for the Lam\'e system it was established in \cite{Duduchava}, p.89, see also \cite{AgrLame}, Proposition 1.8.
Moreover, we know now that $\SF$ is positive. Consequently,  the operator $\AF= \SF^{\frac12}\KF\SF^{-\frac12}=      \ZF \KF \QF$ is
 self-adjoint in $L^2(\G)$.

 This operator $\AF$ is a zero order self-adjoint classical pseudodifferential operator, with the same spectrum as $\KF$. The latter statement is proved in \cite{RozNP1}, Proposition 6.1. By the composition rule for pseudodifferential operators, the principal symbol of $\AF$ is a matrix, similar to the principal symbol of $\KF$,
 \begin{equation*}\label{aF}
    \aF_0(x,\x)=\zF_{\frac12}(x,\x)\kF_0(x,\x)\qF_{-\frac12}(x,\x).
 \end{equation*}
  Being the principal symbol of a self-adjoint operator, the symbol $\aF_0(x,\x)$ has only real eigenvalues; it follows that the eigenvalues of the symbol $\kF_0(x,\x)$, a matrix, similar to $\aF_0(x,\x)$, are the same as the eigenvalues of $\aF_0(x,\x)$; in particular, they are real as well. Moreover,
 we notice immediately that the matrices $\kF_0(\x)$ and $\begin{pmatrix}
                    \bL(\x) & 0 \\
                    0 & 1 \\
                  \end{pmatrix}$ commute. Taking into account the expression for the principal symbols of $\SF^{\frac12}$ and $\SF^{-\frac12}$,
                  we obtain
\begin{equation}\label{SymbolFin}
    \aF_0(\x)=\kF_0(\x)=\frac{i\km}{|\x|}\begin{pmatrix}
                    0 & 0 & -\x_1 \\
                    0 & 0 & -\x_2 \\
                    \x_1 & \x_2 & 0 \\
                  \end{pmatrix}.\end{equation}

                  So, the principal symbols of $\KF$ and $\AF$ coincide; in particular, it follows that $\kF_0$ is a Hermitian matrix. This property follows, of course, also from the fact, easily verified, that the difference, $\KF-\KF^*$, is a pseudodifferential operator of order $-1$.

We cannot declare here that the subsymbol of $\KF$  and  $\ssF_{-1}$ commute; in fact, they do not. However, the important property for our calculations is the following consequence of the composition rule.

\begin{proposition}\label{prop.lower}    Let $\KF$ be a polynomially compact zero order     pseudodifferential operator and for given $\i$, $\pb_\i$ be the polynomial in \eqref{p(o)}. Suppose that $\KF$ is symmetrizable, with the pseudodifferential operator $\SF$ acting as symmetrizer. Then  the principal symbol $\mF_{\i,-1}(x,\x)$ of the operator $\MF_{\i}=\pb_{\i}(\KF)$ is a matrix similar to a Hermitian one, namely to the principal symbol $\bF_{\i,-1}(x,\x)$ of the operator $\BF_{\i}=\pb_{\i}(\AF)$ equals
\begin{equation}\label{sub}
    \bF_{\i,-1}(x,\x)=\zF_{\frac12}(x,\x)\mF_{\i,-1}(x,\x)\qF_{-\frac12}(x,\x);
\end{equation}
here $\qF_{-\frac12} $,  $\zF_{\frac12} $ are principal symbols of operators $\QF=\SF^{\frac12}$, resp., $\ZF=\SF^{-\frac12},$ see \eqref{qF}, \eqref{zF}.
\end{proposition}
\begin{proof} The pseudodifferential  operator $\MF_{\i}$ equals
\begin{equation}\label{similarity}
    \MF_{\i}=\SF^{\frac12}\BF_{\i}\SF^{-\frac12}= \SF^{\frac12}\pb_{\i}(\AF)\SF^{-\frac12},
\end{equation}
 therefore, for the principal symbols of $\MF_\i$ and $\BF_\i$, symbols of order $-1$ we have the equality,
\begin{equation}\label{simil.symbol}
    \mF_{\i,-1}=\qF_{-\frac12} \bF_{\i,-1} \zF_{\frac12}.
\end{equation}
It remains to recall that $\bF_{\i,-1}$ is a principal symbol of a self-adjoint operator, and therefore it is a Hermitian matrix.
\end{proof}

\subsection{General asymptotic formulas}

The operator $\MF_{\i}$ in \eqref{similarity} is a pseudodifferential operator of order $-1$, and its principal symbol, a $3\times 3$ matrix with real eigenvalues, is denoted by $\mF_{\i}(x,\x)$ (sometimes the subscript $-1$ is added in this notation in order to recall that this is a symbol of order $-1$.) Having this symbol at hand, the asymptotics of the eigenvalues of $\KF$, tending to $\o_\i$, can be found by the general result obtained in \cite{RozNP1}, Theor. 6.2.
 We reproduce here this Theorem, adapted to our particular case.

First, we introduce proper notations. For a point $\o_\i$ in the essential spectrum of $\KF$,  $n_{\pm}(\KF; \o_{\i},  \cdot)$ denotes the counting functions of eigenvalues of the operator $\KF$ in intervals near  $\o_{\i}$. Namely, we fix some reference points  $\t_{\pm}$  in such way that the interval $(\o_\i-\t_-, \o_\i+\t_+)$ contains no points of the essential spectrum of $\KF$, other than $\o_\i$. Then, for a small $\t>0,$ $n_{\pm}(\KF;\o_{\i}, \t)$ denotes the number of eigenvalues of $\KF$ in the interval $(\o_\i+\t, \o_\i+\t_+)$ for the '+' sign, resp., $(\o_\i-\t_-, \o_\i-\t)$ for the '-' sign. The object of our study is the behavior of $n_{\pm}(\KF;\o_{\i}, \t)$ as $\t\to 0$. It stands to reason that  $n_{\pm}(\KF;\o_{\i},\t)$ is bounded if there are only finitely many eigenvalues of $\KF$ in the corresponding (upper or lower) neighborhood of $\o_\i.$ On the other hand,
 if there are infinitely many eigenvalues of $\KF$ in such neighbourhood then these eigenvalues must converge to $\o_\i$ and therefore $n_{\pm}(\KF;\o_{\i}, \t)$ tends to infinity as $\t\to 0$. The change in the choice of the reference points $\t_{\pm}$ does not influence the rate of growth of $n_{\pm}(\KF;\o_{\i}, \t)$ as $\t\to 0$, therefore they are not reflected in our notations.

   Next, for a diagonalizable matrix $\mF$ with real eigenvalues, the expression $\Tr_\pm^{(2)}(\mF)$ denotes the \emph{sum of squares} of positive, resp.,  negative, eigenvalues of the matrix $\mF.$ Finally, $\pmb{\pmb{\o}}$ denotes the 1-form
$\pmb{\pmb{\o}}=\x_1d\x_2-\x_2d\x_1=d\theta$ in the polar co-ordinates on the unit circle $S^1\subset \texttt{T}_x \G$, $\x_1=\cos\theta, \x_2=\sin\theta$, $0\le\theta\le 2\pi$ and $d\Sc(x)$ denotes the area element for the surface measure induced by the embedding $\G\subset \R^3.$

\begin{theorem}\label{Th.BS} Let $\BF$ be an order $-1$ self-adjoint pseudodifferential operator on a manifold $\G$ of dimension 2, with principal symbol $\bF(x,\x),$ a Hermitian matrix. Then for the eigenvalues of $\BF$
the asymptotic formulas hold
\begin{equation}\label{BSas}
  n_\pm(\BF, 0,\t)\sim C^{\pm}(\BF,0)\t^{-2}, \t\to 0,
\end{equation}
\begin{equation}\label{Form.BS}
C^{\pm}(\BF,0)=2^{-1}(2\pi)^{-2}
\int_{\mathrm{S}^*\G}\Tr_{\pm}^{(2)}(\bF(x,\x))\pmb{\pmb{\o}} d\Sc(x),
\end{equation}
\end{theorem}
\noindent(note that the zero in the notation  $n_\pm(\BF, 0,\t)$ refers to zero being the only point of essential spectrum for the compact operator $\BF.$)
By the results of \cite{RozNP1}, for the zero order NP operator $\KF$, the asymptotics of $n_{\pm}(\KF; \o_{\i},  \cdot)$ is determined by the asymptotics of eigenvalues of the order $-1$ operator $\MF_{\i}=\pb_{\i}(\KF)$ with the principal symbol $\mF_{\i}(x,\x)$.  Since the operator $\MF_{\i}$ is symmetrizable, it has the same eigenvalues as the self-adjoint operator $\BF_{\i}$. The principal symbol
 $\bF_{\i}(x,\x)$ of $\BF_{\i}$ is a Hermitian matrix with the same eigenvalues as $\mF_{\i}(x,\x)$. Therefore, Theorem \ref{Th.BS} can be made concrete in the following way.

 \begin{theorem}\label{GeneralAs.Theorem} Let $\KF$ be the elastic NP operator on a smooth closed surface $\G\subset \R^3.$ Then the asymptotics of eigenvalues of $\KF$, converging to the point $\o_{\i}$ of the essential spectrum, is described by the formula
 \begin{equation}\label{Formula3}
    n_{\pm}(\KF;\o_\i, \t)\sim C^{\pm}(\o_\i) \t^{-2}, \t\to +0,
    \end{equation}

    \begin{equation}\label{Formula4}
      C^{\pm}(\o_\i)=2^{-1}(2\pi)^{-2}\int_{\mathrm{S}^*\G}\Tr_{\pm}^{(2)}(\mF_{\i}(x,\x))\pmb{\pmb{\o}} d\Sc(x),
\end{equation}
where $\mF_\i=\mF_{\i,-1}$ is the principal symbol of the order $-1$ operator $\MF_\i=\pb_\i(\KF),$ see \eqref{complete symbol}.
\end{theorem}

In this way, the asymptotic formula \eqref{Formula3}, \eqref{Formula4} is derived, as a special case, for our  particular values of the dimension of the manifold and the order of the operator, of general results by M.Sh. Birman and M.Z. Solomyak, \cite{BS} and \cite{BSSib}.  The English translation of the (rather technical) proof in \cite{BS} is virtually inaccessible now, but an interested Reader can be directed  to a recent, soft analysis,  proof by R. Ponge \cite{Ponge}. This proof concerns a more restricted version of the general result in \cite{BS}, for a smooth surface and a homogeneous symbol, however it covers our situation.  This result can also be found in \cite{Ivrii}, however  the proof  (based upon the general theory developed in this book) is left there as an exercise to readers.

There is a simple but  important consequence of this general result.
\begin{proposition}\label{prop.sign}
  Suppose that at some point $(x_0,\x_0)\in \mathrm{S}^*(\G)$, one of eigenvalues of the matrix $\mF_\i(x_0,\x_0)$ is positive (negative). Then there exists an infinite sequence of eigenvalues of $\KF$ converging to $\o_\i$ from above (from below) satisfying the asymptotic formula \eqref{Formula3}, \eqref{Formula4} (with the corresponding sign), with a nonzero coefficient $C^{\pm}(\o_\i)$.
\end{proposition}
\begin{proof}Eigenvalues of a continuous diagonalizable matrix-function of $(x,\x)$ depend continuously on the variables $(x,\x)$. Therefore the eigenvalue in question is positive (negative) in a neighborhood of $(x_0,\x_0)$. Consequently, the integrand in \eqref{Formula4} with proper sign is positive on a set of positive measure. This means that the integral determining $C^{\pm}(\o_\i)$ is positive.
\end{proof}

Note that Proposition \ref{prop.sign} uses, in fact, only  that the coefficients $C^{\pm}(\o_\i)$ are integrals of sign-definite microlocal expressions involving the symbol, and not on the particular form of \eqref{Formula4}. In the next Section we express the conditions of Proposition \ref{prop.sign} in geometrical terms.

Formally the basic theorem by M.Birman and M.Solomyak, as  formulated in \cite{BS}, \cite{BSSib},  concerns only connected surfaces. However, obviously, in the case of  the surface $\G$ consisting of several connected components, $\G=\cup \G_\vk,$ which happens when the body $\Dc$ possesses cavities, the Neumann-Poincar\'e operator $\KF$ is the direct sum of pseudodifferential operators $\KF_{\vk}$ on the components $\G_\vk.$ The discrete spectrum of such sum is the union of the spectra of $\KF_{\vk}$.

\subsection{The two-sided eigenvalue  asymptotics}\label{discussion}
Although the expression in \eqref{Formula4} depends formally on the choice of local co-ordinates
on the surface $\G$ and the frame in the fiber, the integral in \eqref{Formula4}  and, moreover, the  $\pmb{\pmb{\o}}$-integral in \eqref{Formula4} over the cotangent  circle  are invariant under the change of an orthogonal co-ordinate system on $\G$ and under the change of the frame, as was established in \cite{RozNP1}.

The expression under the integral in \eqref{Formula4} is rather hard to handle. In fact, for the matrix $\mF_{\i,-1}$ depending on $x,\x$ -- as we see later, determined by the material characteristics and by the principal curvatures -- one needs to calculate the eigenvalues, separately the positive and negative ones, and then integrate over $\mathrm{S}^*(\G)$ certain expressions containing the squares of positive, resp., negative, eigenvalues. It seems that the task of finding treatable analytical expressions here is very hard. On the first step above, namely, finding the eigenvalues of a $3\times 3$ symbolic matrix, this means, solving a third degree algebraic equation with symbolic coefficients, is  rather hard. This equation has  three real roots, their expression involves cubic roots from complex numbers -- and formulas \eqref{Formula4} require further separating positive and negative ones, and then integrate the result -- altogether, this is highly   impractical. Formulas \eqref{Formula3}, \eqref{Formula4} may serve, probably, for numerical calculations as well as for the asymptotic analysis. These formulas can be, however, used for obtaining qualitative results for the properties of the NP eigenvalues under certain geometrical conditions, see Sections 4,5,6.  What we can, however, achieve more easily, is to find the explicit asymptotics for the distribution function for the \emph{union} of the sequences of eigenvalues converging  to  $\o_\i$, in other terms, for the sum of functions  $n_{\pm}(\KF;\o_\i,\t)$ counting the eigenvalues lying below and above $\o_\i$. It follows from the formulas for the asymptotics of the absolute values   of the  eigenvalues of  an order $-1$ pseudodifferential operator to which the spectral problem for the NP operator is reduced. This possibility is explained by a very fortunate relations between the order of the operator involved and of the dimension of the space. Such relation leads to the exponent in \eqref{Formula4} (which, in the general case, equals the dimension of the manifold divided by  minus the order of the operator) to be equal to  $2.$ And, fortunately, since the eigenvalues of the symbol $\mF_\i$ are real, it is  the trace of the square or the matrix symbol that expresses  the sum of squares of its eigenvalues  (an effect of similar kind is present also if we are interested in the sum of some even integer powers of the eigenvalues of the matrix, with a more complicated but still polynomial expression.)   Therefore, Theorem \ref{GeneralAs.Theorem}
has as a  consequence the following asymptotic formula.

\begin{theorem}\label{SingNumbers} Let the conditions of Theorem \ref{GeneralAs.Theorem} be satisfied and let $\mF_\i=\mF_{\i,-1}(x,\x)$ be the principal symbol of the operator $\MF_\i=\pb_{\i}(\KF).$ Then for the eigenvalues of $\KF$ the following asymptotics holds:
\begin{equation}\label{Form.l pm}
 n(\KF;\o_\i,\t)\equiv   n_+(\KF;\o_\i,\t)+n_-(\KF;\o_\i,\t)\sim C(\o_\i)\t^{-2},\, \t\to 0,
\end{equation}
where the coefficient $C(\o_\i)$ equals
\begin{equation}\label{coeff.pm}
  C(\o_\i)=C^+(\o_\i)+C^-(\o_\i)=2^{-1}(2\pi)^{-2}\int_{\mathrm{S}^*\G} \Tr (\mF_{\i,-1}(x,\x)^2) \pmb{\pmb{\o}} d\Sc (x).
\end{equation}
\end{theorem}
The geometrical meaning of the expression in \eqref{coeff.pm} will be discussed further on, in Section 4.

Thus, the problem remains of calculating the symbol $\mF_{\i,-1}(x,\x)$ (or its positive and negative eigenvalues) in \emph{some} co-ordinate system and frame. This freedom will be used essentially in the reasoning to follow.

In \cite{RozNP1}, when discussing the reduction of the spectral problem for a polynomially compact pseudodifferential  operator to the one for a compact operator, we considered also the case when the operator  $\mF_{\i,-1}$ has  everywhere vanishing principal symbol, $\mF_{\i,-1}(x,\x)\equiv 0.$ In this exceptional case, the coefficient in front of $\t^{-2}$ in the asymptotic formula \eqref{Formula3} vanishes for both signs; in this way \eqref{Formula3}  becomes non-informative.  Such case was called 'degenerate' in \cite{RozNP1}, and obtaining eigenvalue asymptotic formulas required  additional considerations, see Lemma 5.1 and Theorem 5.2 in \cite{RozNP1}.  However, in the only explicitly calculated case of the sphere, see \eqref{ball},  the results can be expressed as
\begin{gather}\label{ball.counting}
    n_{+}(\KF;0,\t)\sim \frac{9}{16}\t^{-2}, \,\t\to 0;\\\nonumber
    n_{-}(\KF; 0,\t)=0,\,;\\\nonumber
    n_{+}(\KF; -\km,\t)\sim(4\km)^2\t^{-2}, \t\to 0;\\\nonumber
    n_{-}(\KF; -\km,\t)=0; \\\nonumber
    n_+(\KF;\km,\t)\sim (4\km)^2\t^{-2},\t\to 0;\\\nonumber
    n_-(\KF;\km,\t)=0.
    \end{gather}
so this case is non-degenerate. We  will see  further on in this study that for the NP operator, the nondegenerate case  always occurs.

\section{The structure of the symbol $\mF_{\i,-1}$. Geometry considerations}\label{Sect.structure}

\subsection{Formulation}\label{formulation Structure}
Due to the results described in Sect.\ref{Sect.symm}, in order to find the coefficients in the asymptotic formula for eigenvalues, we need to calculate the effective principal symbols $\mF_{\i}(x,\x)\equiv\mF_{\i, -1}(x,\x),$ $\i=-1,0,1$, of the  order $-1$ pseudodifferential operators $\MF_{\i}=\pb_{\i}(\KF)$.  We aim now for avoiding the (very tedious) direct calculation of these symbols and their bulky and unwieldy expression using some \emph{a prori} symmetry properties. More detailed calculations follow in Sect.6.

  We have already started to investigate the  structure of the symbol $\mF_{\i}$ in Section 2. Here we are going to determine the character of dependence of this symbol on the geometrical characteristics of the surface. We denote by $\kb_1(\xb), \kb_2(\xb)$ the principal curvatures of the surface $\G$ at a point $\xb\in\G$. If a parametrization of $\G$ is chosen, the notation $\kb_1(x), \kb_2(x)$ is used, $x\in\O\subset \R^2$, as long as it does not cause a misunderstanding. Recall that the product $\kb_1(x) \kb_2(x)$ is the Gaussian curvature, $\frac{\kb_1(x)+\kb_2(x)}2$ is the mean curvature and $\left(\frac{\kb_1(x)+ \kb_2(x)}2\right)^2$ is the Willmore curvature of the surface. Note also that we have taken the co-ordinate $x_3$  directed along  the outward normal to $\G$, therefore the principal curvatures are \emph{negative}  at the points where the surface is convex.

We will call in this section a matrix-function $M(\x)$ \emph{universal}  if if it, in standard co-ordinates and frame, depends only on the covector $\x$ and the Lam\'e constants, but does not depend on the point on $\G$ where is is calculated. Later, in Sect.6, the term 'universal' will refere to functions which depend only on $\x.$

The crucial property we establish in this Section is the following.
\begin{theorem}\label{Result}
For any point $\xb^{\circ}\in \G,$ in the standardly chosen co-ordinates  on the surface $\G$ in a neighborhood of $\xbc$ and the corresponding frame in $\C^3$, the symbol $\mF_{\i}(x,\x)$ has at the point $\xbc$ the structure
\begin{equation}\label{SymbolStructure}
    \mF_{\i}(x,\x)=\kb_1(\xbc)M^{(1)}_{\i}(\x_1,\x_2)+\kb_2(\xbc)M_{\i}^{(2)}(\x_1,\x_2),
\end{equation}
with  universal matrices $M_{\i}^{(1)}(\x_1,\x_2)$, $M_{\i}^{(2)}(\x_1,\x_2)$, order $-1$ positively homogeneous in $\x$, depending on the Lam\'e constants $\l,\m$ but not depending on the surface $\G$.
\end{theorem}

We would like to stress that the representation \eqref{SymbolStructure} is valid only in the specially selected co-ordinates system and frame at $\xbc$. These are chosen depending on the geometry of $\G$ near the point $\xbc$. However, recall, the eigenvalues of the symbol $\mF_{\i}(x,\x)$ do not depend on the co-ordinates chosen or the frame, so the integrand in \eqref{Formula4} is invariant under these changes.

We  present the proof of Theorem \ref{Result} further on in this section.
\subsection{C--co-ordinate systems}\label{subsect coord}
We choose near a point $\xb^\circ\in\G$ a special co-ordinate system, where the structure of the symbol is more treatable. We will call it the 'C--co-ordinates at $\xb^\circ$'. It is in this system that the representation \eqref{SymbolStructure} is valid.

Suppose first that $\xb^\circ$ is an umbilical point of the surface $\G$. Recall that a point on a smooth  surface  in $\R^3$ is called umbilical if the principle curvatures at this point coincide \begin{footnote}{It is a long-standing problem, the Carath\'eodory conjecture, concerning the minimal possible number of umbilical points on a surface. According to this conjecture, for any smooth closed surface with nonzero Euler characteristic, there must exist at least two umbilical points. The topological torus may have no umbilical points at all. For surfaces of a different topological type, the existence of at least one umbilical  point follows from simple topological considerations. For  analytic surfaces the conjecture was settled not long ago, see \cite{Gutierrez} and \cite{Ivanov} and the literature cited there. For a finite smoothness, and even for the infinite non-analytic one, the conjecture seems to be still unresolved.}\end{footnote}. For such a point, we direct the orthogonal  $x_1,x_2$ axes arbitrarily  in the tangent plane to $\G$ at $\xb^{\circ}$ and direct the  $x_3$ axis orthogonally to them, in the \emph{outward} direction at $\xb^{\circ}$.

If $\xb^{\circ}$ is \emph{not} an umbilical point, we direct $x_1,x_2$ axes along the lines of principal curvatures of $\G$ at $\xb^{\circ}$ and direct the $x_3$ axis along the outward normal at $\xb^{\circ}$.

In both cases, the surface $\G$ near $\xb^{\circ}$ is described by the equation $x_3=F(x_1,x_2)\equiv F(x)$ with $F(0,0)=0,$ so $\xbc=(0,0,0),$ and
\begin{equation}\label{Function F}
     \nabla F(0,0)=0, F(x_1,x_2)=\frac12(\kb_1(\xb^{\circ})x_1^2+\kb_2(\xb^{\circ})x_2^2)+O((x_1^2+x_2^2)^{3/2}),
\end{equation}
where $\kb_1(\xb^{\circ}),\kb_2(\xb^{\circ})$ are the principal curvatures of $\G$ at $\xb^{\circ},$  while $\kb_1(\xb^{\circ})=\kb_2(\xb^{\circ})$ at an umbilical point.  Note that in the non-umbilical case, the numbering of the principal curvatures matches the numbering of $x_1,x_2$ co-ordinates. This co-ordinate system will be called C--co-ordinates. The dual co-ordinates $\x=(\x_1,\x_2)$ in the cotangent space are directed along the same axes. The frame in the fiber $\R^3$ at the point $\xbc$ is chosen along the axes $x_1,x_2,x_3$.

In a neighborhood of a non-umbilical point, the curvature lines are smooth and the  C--co-ordinate systems chosen above depend smoothly on the point $\xb^{\circ}$. On the other hand, near an umbilical point, the curvature lines may behave  rather wildly, and the above co-ordinates system can depend on the base point fairly non-smoothly. Since we will need further on to trace the behavior of symbols under the change of the starting point $\xbc$,
we   adopt certain co-ordinate systems at points $\xt\in\G$ near $\xb^{\circ}$, arbitrarily, but consistent smoothly  with the C--co-ordinate system at $\xb^{\circ}$. Namely, for a point $\xt\in\G$ with co-ordinates $(x_1,x_2,x_3)\equiv(x,x_3(=F(x)))$ with respect to the C-co-ordinate system at $\xbc$, we consider the projection $P_{\xt}$ of the tangent plane at $\xb^{\circ}$,  $\mathrm{T}_{\xb^\circ}(\G)$, to the tangent plane  $\mathrm{T}_{\xt}(\G)$. The co-ordinates $y=(y_1,y_2)$ on $\mathrm{T}_{\xt}(\G)$ will be generated on $\mathrm{T}_{\xt}(\G)$ from $\mathrm{T}_{\xb^\circ}(\G)$  by this projection, with $y_3$ axis directed along the exterior normal at $\xt$ to $\G$. What follows from this construction, is that the Jacobi matrix of  this co-ordinate  transformation is, up to higher order terms as $\xt\to\xb^{\circ}$, the identity matrix $\Eb$ plus a term linear in first order derivatives of $F$. The first derivatives of this Jacobi matrix, by the chain rule, are matrices, linearly depending on the second derivatives of $F$.  Therefore, due to \eqref{Function F}, the derivatives of the Jacobi matrix  are the principal curvatures $\kb_1(\xb^{\circ}),\kb_2(\xb^{\circ})$  at $\xb^{\circ}$, which enter with universal coefficients. We direct the vectors in the frame of the fibre of the bundle at $\xt$ along the co-ordinate axes. The derivatives of the transformation matrix of the fiber to the standard frame at $\xb^{\circ}$   are, again, linear forms of  $\kb_1(\xb^{\circ}),\kb_2(\xb^{\circ})$ with universal coefficients, up to some higher order terms, as $\xt\to\xb^{\circ}$.

\subsection{The symbol $\mF_\i$ and the principal curvatures.}
The key point in the reasoning to follow is the fact that the symbols we obtain are linear forms of principal curvatures with coefficients depending on the co-variables $\x=(\x_1,\x_2)$ and the Lam\'e constants $\l,\m$ but not on the point $\xb^{\circ}$. We call such coefficients 'universal'.  Our considerations will be based upon the analysis of the structure of various terms in the expansion of the principal symbol $\mF_{\i,-1}$. The general idea is the following. If we have a function $\F(z),$ depending on some parameters $\l,\m,\x$, and $z=F(x)$ is a function on $x\in\R^d$ such that $F(0)=0, \nabla F(0)=0,$ then, according  to the chain rule,  the iterated gradient $\nabla_{x}^2 \F(F(x))|_{x=0}$ is a linear (matrix) form of second derivatives of $F$ at zero, with coefficients depending only on $\l,\m,\x$ and $\F''_z(0),$ this means with universal coefficients.

We pass to the study of the symbol $\mF_\i(x,\x)$ (the principal, order $-1$, symbol of the operator $\MF_{\i}$). Recall that this symbol is constructed following the rules \eqref{type1}, \eqref{type2}, \eqref{complete symbol}.
The expression \eqref{complete symbol} is a sum of 25 terms.
These terms involve the principal symbol $\kF_0$ of the operator $\KF$,  its first order derivatives in $x$ and $\x$ and, finally, the
subsymbol $\kF_{-1}$ of $\KF$.  It was explained in Sect.2 that relations \eqref{type1}, \eqref{type2}, \eqref{complete symbol} show that each of 25 additive terms  in the expression for   $\mF_{\i}$  contains only one factor of  order $-1$, all the remaining factors having order $0$. Since we may perform our calculations in any co-ordinate system by our choice, we will study $\mF_{\i}$ in the C--co-ordinate system centered at the point $\xb^{\circ}$.

The symbol $\kF_0$ does not depend on the geometry of $\G$, as can be seen in \eqref{NPPrinc}. So, the only way how $\mF_{\i}$ can depend on the geometry is via $\nabla_x \kF_0, \nabla_{\x}\kF_0,$ and $\kF_{-1}.$
\subsection{Dependence on the geometry of $\G$. 1. $\nabla_x \kF_0, \nabla_{\x}\kF_0$}\label{subsection.deriv.Princ}
Here and in the next subsection we determine which characteristics of the surface  may be present in the expression of the terms in \eqref{complete symbol}. The unwieldy explicit formulas are not needed at the moment (they will be discussed in Sect.6 in more detail).  First, we can see in \eqref{NPPrinc}, that the expression for the  symbol $\kF_0$ does not involve any dependence  on $\xbc$, therefore, the same is true for the $\x$ derivatives of $\kF_0$ (these derivatives can be calculated directly from \eqref{NPPrinc}, but we will not do this at the moment).

To evaluate $\nabla_{x}\kF_0$ at the point $\xb^{\circ}\in\G$, we take another point $\xt\in\G$, in a neighborhood of $\xb^{\circ}$. We consider also the C--co-ordinate system centered at $\xb^{\circ}$ and the consistent system centered at $\xt$,  as explained in Section \ref{subsect coord}. We will mark by the superscript $\circ$ the principal symbol $\kF_0$ and other objects expressed in the $\xb^{\circ}$-- centered system and by $\bullet$ the same objects, but expressed in the $\xt$- centered system.

In this notation, we are interested in the derivative $\nabla_{x}\kF_0^{\circ}(x,\x)$ calculated at the point  $\xb^{\circ}$,   i.e., at $x=0$. Thus we study the behavior of the principal symbol as $\xb^{\bullet}$ approaches $\xb^\circ$.  We denote by
$Z=Z_{\xb^{\bullet}}$ the variables change  on $\G$ in a  neighborhood of $\xb^\circ$ from the $\xb^\circ$-centered co-ordinates to the $\xt$-centered  ones. The Jacobi matrix $DZ=DZ_{\xbc}$ of this transformation  contains the first order derivatives of the function $F$  at $\xt$. The transformation $U(\xt)\in GL(\R,3)$ from the $\xb^{\circ}$-frame to the $\xt$-frame depends linearly on the first order derivatives of $F$ at  $\xt$ as well.

 We use now  the classical rule of transformation of the symbol under the change of variables and the natural rule of transformation under  the change of the basis in the fiber. Namely, we write the symbol in $\xt$-centered co-ordinates -- it will have the same form as \eqref{NPPrinc} -- and then transform it to $\xb^{\circ}$-centered   co-ordinates.  In this way, we have at the point $\xt$
\begin{gather*}
  \kF_0^{\circ}(x,\x)  =  U(\xt)\kF^{\bullet}_0(Z(\xb^{\bullet}),((DZ)^{-1})^\top \x)U(\xt)^{-1}=\\ \nonumber U(\xt)\frac{i\km}{|\x|}\begin{pmatrix}
                    0 & 0 & -\y_1 \\
                    0 & 0 & -\y_2 \\
                    \y_1 & \y_2 & 0 \\
                  \end{pmatrix}U(\xt)^{-1},
\end{gather*}
 with $\y\equiv(\y_1,\y_2)=((DZ)^{-1})^\top \x.$ We recall here that the variables change $Z(\xt)$, its differential $DZ(\xt)$, and the linear transformations $U(\xt)$ depend smoothly on the first order derivatives of the function $F(x_1,x_2)$, moreover they become identity maps as $\xt\to\xb^{\circ}$, since the derivatives of $F$ vanish ay $\xb^\circ$. Therefore, the derivatives of $\kF_0^{\circ}(x,\x)$ at $\xb^\circ$, by the chain rule, depend linearly on the second derivatives of $F$ at $\xb^\circ$, with no more characteristics of $F$ involved. Since the co-ordinates $x_1,x_2$ have been chosen along the curvature lines of $\G$ at $\xb^{\circ}$, the mixed second derivative of $F$ vanishes, while the pure second derivatives are equal to the principal curvatures of the surface at the point $\xb^\circ$, moreover, this dependence is linear. Thus, we have established that
 \begin{equation*}
    \nabla_{x}\kF_0^{\circ}(x,\x)_{x=0}=\km(\kb_1(\xb^{\circ})f_1(\x)+\kb_2(\xb^{\circ})f_2(\x))
 \end{equation*}
 in the C--co-ordinate system centered at $\xb^{\circ}$, with some (matrix) symbols $f_1,f_2$ of order $0$ depending on $\x$ (and the Lam\'e constants) only. The same conclusion holds at umbilical points, where $\kb_1(\xb^{\circ})=\kb_1(\xb^{\circ})$, by a similar reasoning.

\subsection{Dependence on the geometry of $\G$. 2. $\kF_{-1}(x,\x)$}\label{subsection.subsymbol}

 In order to find the required representation for the subsymbol $\kF_{-1}(x,\x)$ of the operator $\KF$, it is more convenient to  work not with the symbol but with  the kernel of  the integral operator.

We consider  the local expression \eqref{kernelNP} for the NP operator.  Having the point $\xb=\xb^\circ $ (the point $x=0$ in the  C-co-ordinate system centered at $\xbc)$ fixed, we expand all entries of the  kernel $\Kc$ in the asymptotic  series in terms, positively homogeneous in $y-x$.  We are interested in the first two terms in this expansion in the form
\begin{gather}\label{KrnelExpnsion}
     \Kc(\xb,\yb)= \Kc_0(x,x-y)+\Kc_{-1}(x,x-y) +O(1); \\ \nonumber  \xb=(x,F(x))=0\in \G,\,  \yb=(y,F(y))\in\G,
\end{gather}
where $\Kc_0(x,x-y)$ is order $-2$ positively homogeneous and odd in $(x-y)$ and $\Kc_{-1}(x,x-y)$ is order $-1$ positively homogeneous in $(x-y)$.  In order to find these terms, we consider the expansion for separate terms in \eqref{kernelNP}. Here we keep in mind the Taylor expansion for the function $F$ near $\xb^{\circ}$,
 $F(x)=\frac{1}{2}(H(\xb^{\circ}) x,x)+O(|x|^3)$, $x\to 0$, where $H=\diag(\kb_1(\xbc),\kb_2(\xbc))$.   Next, by our choice of co-ordinates,  the co-ordinate axes $x_1,x_2$ lie along the eigenvectors of the matrix $H$ (for  an umbilical point, i.e., when $H$ is a multiple of the unit matrix, any orthogonal directions may be chosen.) In this co-ordinate system, the first fundamental form of the surface $\G$ at $\xb^\circ$ is the identity one,
 \begin{equation*}
 \Ib[\G]_{\xb^\circ}(dx)=|dx|^2.
 \end{equation*}
  The second fundamental form for this surface at $\xb^\circ$ is  diagonal in this co-ordinate system:
  \begin{equation*}
   \Ib\Ib[\G]_{\xb^\circ}(dx)=\kb_1(\xb^\circ)(dx_1)^2+\kb_2(\xb^\circ)(dx_2)^2,
  \end{equation*}
  calculated, recall, with the direction of the normal vector chosen to be the outward one, so it is \emph{negative} at those points where the surface is convex.
 For the entries in the kernel of the integral operator $\KF$, at the point  $\xb^{\circ}$ with co-ordinates $(x,F(x))=(0,0)$ in the chosen co-ordinate system,
 we use the standard relations for co-ordinates in this system. Namely, for the components of the normal vector, we have
 \begin{equation}\label{normal.1}
 \n_\a(y)=\n_\a(0)+\kb_\a(\xb^{\circ}) y_\a+ O(|y|^2)= \kb_\a(\xb^{\circ}) y_\a+ O(|y|^2), \a=1,2,
 \end{equation}
 and
 \begin{equation}\label{normal.2}
   \n_3(y)=1-O(|y|^2).
 \end{equation}
The distance between points, entering in \eqref{kernelNP}, is found as
 \begin{equation}\label{distance}
 |\xb^\circ-\yb|^2=|x-y|^2\left(1+2\frac{\Ib\Ib[\G]_{\xb^\circ}(x-y)^2}{|x-y|^2}\right)+o(|x-y|^4).
 \end{equation}

 We substitute  \eqref{normal.1}, \eqref{normal.2}, \eqref{distance} in \eqref{kernelNP} and obtain that the leading  term $\Kc_{-1}$ in the singularity as $y\to x$ of the kernel $\Kc_{-1}(y,x-y)$  is a  linear function of the principal curvatures.
\subsection{Proof of Theorem \ref{Result}}\label{proofOfStructure}
Finally, we take into account the structure properties of the symbol $\mF(x,\xi)$ as it depends on the principal symbol and  the subsymbol of the operator $\KF$, see \eqref{type1}, \eqref{type2}, \eqref{complete symbol}. At a given point, in C--co-ordinates, for each summand in \eqref{complete symbol}, the factors  $\kF_0$ do not depend on the geometry of $\G$. Exactly one   factor is present in each summand that involves  the geometry of $\G$, and this term is linear in the principal curvatures. All other terms in the products in \eqref{complete symbol} do not involve geometric characteristics of the surface, and therefore are universal matrices (depending, of course, on the material constants and the direction of the covector $\x\in \mathrm{T}_x^*\G$.)  Therefore, each summand, being the product of 5 factors, and further on,  the whole symbol $\mF_{\i}(x,\x)$, depend linearly on the principal curvatures of $\G$, with universal coefficients. This concludes the proof of Theorem \ref{Result}.

\section{Symmetries, reductions, and curvatures in asymptotic formulas}\label{Symmetries}
In this section we start applying the results about  the spectrum of general polynomially compact pseudodifferential operators, obtained in \cite{RozNP1}, see Theorem \ref{GeneralAs.Theorem}, to study the asymptotics of eigenvalues of the NP operator. At this stage we will see that some properties of the spectrum can be derived by means of qualitative considerations, without calculating the symbols explicitly. Later, we will present more detailed spectral properties based upon explicit calculations.
\subsection{Symmetries}

The first property follows from the fact that the symbol $\kF(x,\x)$ should transform in a definite way  as soon as we permute the co-ordinate axes $x_1$ and $x_2$. We choose the principal curvatres in the convenient way. Namely, let $\kb_1=\kb\ne 0, \kb_2=0$. Then, by \eqref{SymbolStructure},

\begin{equation}\label{perm0}
\mF_{\i}(x,\x_1,\x_2)=\kb M_{\i}^{(1)}(\x_1,\x_2).
\end{equation}
Now, let $\kb_1=0,\,\kb_2=\kb.$ The corresponding symbol must be the same as \eqref{perm0}, after the following transformations caused by the permutation of co-ordinates:
\begin{enumerate}
\item the covariables $(\x_1,\x_2)$ must be permuted, $(\x_1,\x_2) \leftrightarrows (\x_2,\x_1); $
\item $M_{\i}^{(1)}$ should be replaced by $M_{\i}^{(2)}$;
\item the frame in $\C^3$ changes, therefore, the first two horizontal rows should be permuted as well as two first columns.  This is obtained by the transformation $\mF\leftrightarrows V^{-1}\mF ,$ where $V$ is
 the linear unitary transformation in $\C^3$ interchanging the first row with the second one,
 i.e., the matrix
 \begin{equation*}
    V=
    \left(
      \begin{array}{ccc}
        0 & 1 & 0 \\
        1 & 0 & 0 \\
        0 & 0 & 1 \\
      \end{array}
    \right),\, V=V^{-1}.
 \end{equation*}
\end{enumerate}
So, we obtain
\begin{equation*}
 \mF_{\i}(x,\x_1,\x_2)= \kb M_{\i}^{(1)}(\x_1,\x_2)=\kb V^{-1}M_{\i}^{(2)}(\x_2,\x_1)V,
\end{equation*}
therefore, if we denote  $\hat{\x}\equiv \widehat{(\x_1,\x_2)}=(\x_2,\x_1)$
\begin{equation}\label{perm.fin}
  M_{\i}^{(2)}(\x_1,\x_2)=V^{-1}M_{\i}^{(1)}(\x_2,\x_1)V.
\end{equation}
As a result, the symbol $\mF_{\i}(x;\x_2,\x_1)$ must depend on \emph{only one} universal matrix $M_\i(\x_1,\x_2)=M_\i^{(1)}(\x_1,\x_2),$
and we obtain, for a general body, with principal curvatures $\kb_1(x), \kb_2(x),$
\begin{equation}\label{perm.superfin}
  \mF_{\i}(x,\x)=\kb_1(x)M_{\i}(\x)+\kb_2(x)V^{-1}M_{\i}(\hat{\x})V.
\end{equation}

We consider now the special case of $\G$ being the unit sphere  $S^2$ in $\R^3$.  All points on $\G$ are umbilical, moreover, $\kb_1(x)=\kb_2(x)=-1$ everywhere on $\G$. We can choose the local (orthogonal) co-ordinates in an arbitrary way. Therefore, the symbol $\mF_\i(x,\x)$ for the sphere equals
\begin{equation}\label{prop.3}
    \mF_{\i}(x,\x)=-M_{\i}(\x)- V^{-1}M_{\i}(\widehat{\x})V.
    \end{equation}
    For the sphere, the eigenvalues of the symbol $ \mF_{\i}(x,\x)$ are the same for all points $x$, and,  accordingly, the integrand in \eqref{Formula4} is independent of the point $x.$  Therefore, the asymptotic formula  \eqref{Formula3}, \eqref{Formula4} gives the following expression for the coefficient $C^{\pm}(\o_\i)$ for the sphere:

\begin{gather}\label{coeff.Sphere}
   C^{\pm}(\o_\i)=2^{-1}(2\pi)^{-2}\int_{\mathrm{S}^*\G}\Tr_{\mp}^{(2)}(M_{\i}(\x)+V^{-1}M_{\i}(\hat{\x})V)\pmb{\pmb{\o}} d\Sc(x)\\\nonumber
   =(2\pi)^{-1}\int_{S^1}\Tr_{\mp}^{(2)}(M_{\i}(\x)+V^{-1}M_{\i}(\hat{\x})V)\pmb{\pmb{\o}}
\end{gather}
(note the sign:  $\Tr_{\mp}^{(2)}$ enters  in the formula, since the curvature $\kb_1=\kb_2$ equals $-1$).

We recall now the explicit formulas for the eigenvalues of the NP operator on the sphere, see \eqref{ball}. These formulas show that the eigenvalues of $\KF$  approach each of three points of the essential spectrum \emph{from above}. Therefore the coefficients $C^-(\o_\i)$ in \eqref{Formula3} vanish.  This,  by Theorem \ref{Result}, means that the integrand in the last line in \eqref{coeff.Sphere} eigenvalue  asymptotic formula  \eqref{Formula4}, $\Tr_{\mp}^{(2)}(M_{\i}(\x)+V^{-1}M_{\i}(\hat{\x})V)$  is always positive for $-$ sign and it is everywhere zero for the $+$ sign. Therefore, the matrix $(M_{\i}(\x)+V^{-1}M_{\i}(\hat{\x})V)$ for each $\i$, is nonpositive for all $\x\in S^1$, and at least for some $\x$ at least one of eigenvalues is strictly negative.

\subsection{Curvatures and the eigenvalue asymptotics}
The property of $(M_{\i}(\x)+V^{-1}M_{\i}(\hat{\x})V)$ being non-positive does not automatically imply that the matrix $M_{\i}(\x)$ is non-positive -- it is easy to construct a counterexample.  Nevertheless, certain important results can be derived from it.

First of all, we note that the trace $\tr(M_{\i}(\x)+V^{-1}M_{\i}(\hat{\x})V)$ is non-positive for all $\x\in S^1$ and strictly negative for some $\x$. Since $\tr(V^{-1}M_{\i}(\hat{\x})V))=\tr M_{\i}(\hat{\x}),$ we have
\begin{equation*}
\tr M_{\i}(\x)+\tr M_{\i}(\hat{\x})\le 0.
\end{equation*}
 and at least at one point $\x\in S^1$ this trace is negative. It follows that for such $\x$ at least one of eigenvalues of the matrix $M_{\i}(\x)$ is strictly negative.

Now we use the fact that for any bounded body $\Dc\subset\R^3$ with smooth boundary $\G$ there exists a point $\xb^\circ$ where \emph{both principal curvatures are strictly  negative} -- see Section \ref{geom.subsect}.
 Then the trace
\begin{equation}\label{trace}
  \tr(\mF_{\i}(x,\x))=\kb_1(\xb^\circ)\tr(M_{\i}^{(1)}(\x))+\kb_2(\xb^\circ)\tr(M_{\i}^{(2)}(\x))
\end{equation}
is  strictly positive for some $\x$ and, consequently, at least one of eigenvalues of $\mF_{\i}(x,\x)$ is strictly negative. Therefore, as explained in Proposition \ref{prop.sign}, the coefficient $C_{\i}^+$ does not vanish.  Since on every compact smooth surface in $\R^3$, such point $\xb^\circ$ exists, this gives us the following result for the eigenvalues of the NP operator.

\begin{theorem}\label{thm.Convex} For any body $\Dc$, the coefficients $C^+(\o_\i)$ in \eqref{Formula4} are strictly positive; therefore there exist infinitely many eigenvalues of the NP operator $\KF$ approaching the points $\o_\i$ of the essential spectrum \emph{from above}, and they satisfy the asymptotic law \eqref{Formula3}, \eqref{Formula4} with 'plus' sign.
\end{theorem}

On the other hand, if there exists a point at the boundary where it is concave, the existence of eigenvalues converging to $\o_\i$ from below is granted. More exactly,

\begin{theorem}\label{thm.concave}
  Let the boundary $\G$ of the body $\Dc$ have at least one point  such that both principal curvatures are non-negative while at least one is strictly positive. Then the coefficients $C^{-}(\o_\i)$
 in \eqref{Formula4} is positive, this means that there are infinitely many eigenvalues of $\KF$ approaching $\o_\i$ from below, and they satisfy the asymptotic formula \eqref{Formula3} with '-' sign.
\end{theorem}

An interesting particular case  of this theorem is the following.

\begin{corollary}\label{cor.cavity}Suppose that the body $\Dc$ contains a cavity inside.  Then the coefficients $C^{-}(\o_\i)$
 in \eqref{Formula4} are positive, there are infinitely many eigenvalues of $\KF$ approaching $\o_\i$ from below, and they satisfy the asymptotic formula \eqref{Formula3} with '-' sign.
\end{corollary}

It is clear that at the point where the boundary of the cavity is convex, the surface $\G$, considered as the boundary of $\Dc$, is concave and Theorem \ref{thm.concave} applies.

\subsection{Two-sided asymptotics}
Using our results on the structure of the symbol $\mF_\i(x,\x)$ we can find a visual expression for the two-sided asymptotics of eigenvalues of the NP operator $\KF,$ following Theorem \ref{SingNumbers}.

Let $\G$ be a smooth compact surface and $\kb_1(\xb), \kb_2(\xb)$ be the principal curvatures at the point $\xb\in\G$. We suppose that they are calculated in C-co-ordinate systems, discussed above, where $d\Sc(x)$ equals the area element for the measure on $\G$ induced by the Lebesgue measure in $\R^3$.

We calculate the integrand  in \eqref{coeff.pm}; it gives

\begin{gather*}\label{snumbers}
  \tr((\mF_\i(x,\x))^2)=\tr([\kb_1(x)M_{\i}{\x})+\kb_2(x)V^{-1}M_{\i}(\hat{\x})V]^2)= \\\nonumber
  \kb_1(x)^2\tr M_\i^2(\x)+2\kb_1(x)\kb_2(x)\tr[M_\i(\x)V^{-1}M_{\i}(\hat{\x})V]+\kb_2^2\tr M_{\i}^2(\hat{\x}),
\end{gather*}
(here, we used the fact  that $\tr((V^{-1}M V)^2)=\tr(M^2)$).
We substitute \eqref{snumbers} into  \eqref{coeff.pm} and use the fact that $\int_{S^1} M_{\i}^2(\hat{\x})\pmb{\pmb{\o}}=\int_{S^1} M_{\i}^2({\x})\pmb{\pmb{\o}}$ and obtain

\begin{gather}\label{two-sided}
   C(\o_\i)=2^{-1} \int_\G(\kb_1^2+\kb_2^2)d\Sc(x)\int_{S^1} \tr (M_\i^2(\x))\pmb{\pmb{\o}}\\\nonumber
  +\int_\G \kb_1(x)\kb_2(x)d\Sc(x)\int \tr[M_\i(\x)V^{-1}M_{\i}(\hat{\x})V]\pmb{\pmb{\o}}\equiv \\\nonumber
  \As_{\i} W(\G)+\Bs_{\i} \chi(\G),
\end{gather}

where $\chi(\G)$ is the Euler characteristic of the surface $\G,$
\begin{equation*}
 \chi(\G)=(2\pi)^{-1}\int_\G \kb_1(x)\kb_2(x)d\Sc(x),
\end{equation*}
by the Gauss-Bonnet formula,
and $W(\G)$ is the Willmore energy of $\G,$
\begin{equation*}
  W(\G)=\int_{\G}\frac{(\kb_1(x)+\kb_2(x))^2}{4}d \Sc(x).
\end{equation*}
The coefficients $\As_{\i}, \Bs_{\i}$ depend only on the Lam\'e constants $\l,\m$ and they are  equal to
\begin{equation*}
  \As_{\i}=2\int_{S^1}\tr (M_\i^2(\x)) \pmb{\pmb{\o}},
\end{equation*}
and
\begin{equation*}
  \Bs_{\i}=2\pi\int_{S^1}\left[\tr[M_\i(\x)V^{-1}M_{\i}(\hat{\x})V]-\tr[M_\i^2(\x)] \right]\pmb{\pmb{\o}}
 \end{equation*}

 We formulate this, rather esthetic, result as a theorem.

 \begin{theorem}\label{thm.twosided}
   Let $\G$ be a compact smooth surface in $\R^3$. Then for each point $\km_\i,\,\i=-1,0,1,$ of the NP operator,
   there exist infinite sequences of eigenvalues converging to $\km_{\i}$ satisfying the two-sided asymptotic law \eqref{Form.l pm} with nonzero coefficients given by \eqref{two-sided}. In particular, the degenerate case in the spectral analysis of polynomially compact pseudodifferential operators never happens for the elastic NP operator.
 \end{theorem}

\begin{remark}One can note a similarity of our Theorem \ref{thm.twosided} and, especially, formula \eqref{two-sided}, with the results of the papers \cite{M}, \cite{MR3D} concerning the eigenvalues, tending to zero, of the compact  NP operator in 3D electrostatics. There, the asymptotics of the, separately considered, positive and negative eigenvalues contains, respectively, integral of positive and negative parts of a certain rather complicated expression involving the curvatures of the surface, while the coefficient in the two-sided asymptotics of these  eigenvalues is expressed in a linear way via the global geometric characteristics of the surface, namely, its Euler charateristic and Willmore energy.
\end{remark}

\subsection{Some geometry}\label{geom.subsect} In the considerations above we used the following geometrical fact: for a smooth closed surface $\G\subset\R^3$, there exists a point $\xb^{\circ}\in\G$ such that the surface is strictly convex at $\xb^{\circ}$, this means, both principal curvatures of $\G$ at $\xb^{\circ}$ are negative. We were not able to locate a proof of this, probably,folklore, result in the literature, therefore we give an elementary proof here (not pretending that it is a novel one).

Let $\db=\diam(\G)$ be the diameter of $\G,$ the largest distance between a pair of points in $\G$. By compactness, such pair must exist (probably, not a unique one, but this does not matter.) Let $A,B$ be such points.
Let $\kb_1(B),\kb_2(B)$ be the principal curvatures of $\G$ at $B$ in a C-co-ordinate system.
\begin{proposition}\label{Prop.convex} The principal curvatures at the point $B$ satisfy $\kb_1(B),\kb_2(B)
\le -\db^{-1}$.
\end{proposition}
\begin{proof}In a C-co-ordinate system centered at $B,$ the surface is described near $B$ by
\begin{equation*}
  x_3=\frac{\kb_1(B)}2 x_1^2 + \frac{\kb_2(B)}2 {x_2}^2+O(|x|^3).\end{equation*}
Therefore, the distance from $A$ to the point $\xb$ of $\G$ near $B$, with co-ordinates $x_1,x_2,x_3$
satisfies
\begin{equation*}
  \dist^2(A,\xb)=(x_1^2 +x_2^2+(\db+\frac{\kb_1(B)}2 x_1^2 + \frac{\kb_2(B)}2{x_2}^2))^2+O(|x|^3).
\end{equation*}
If we suppose, for example, that $\kb_1(B)>-\db^{-1}$, we have for $x_2=0,$
\begin{equation*}
  \dist^2(A,\xb)=(x_1^2+(\db+\frac{\kb_1(B)}2 x_1^2))^2+O(x_1^3)= \db^2+x_1^2 -\db\kb_1(B)  x_1^2 +O(x_1^3),
\end{equation*}
therefore, $\dist(A,\xb)>\db$ for small $x_1$, and this contradicts the fact that $\db$ is the diameter of $\G.$
\end{proof}
\section{Explicit calculation of the symbol $\mF_{\i}$. Dependence on the Lam\'e constants}\label{explicit}

As marked above, the expression for the coefficients in \eqref{Formula4} is highly complicated since it involves integration of the eigenvalues of a definite sign of the symbolic matrices $\mF_{\i}(x,\x)$. In this section we demonstrate, nevertheless, how the symbol $\mF_{\i}$ can be calculated. In order to find the symbol $\mF_{\i,-1}(x,\x)$  we need to find just one  matrix-function $M_{\i}(\x)$. This matrix depends on the Lam\'e constants of the material and is a function of the covector $\x\in \R^2.$

\subsection{The standard surface}\label{Sect.Stand}

Since the universal symbol $M_{\i}(\x)$ is the same for all surfaces, it is sufficient to find it for just one, specially chosen,  surface, where the calculation of $\m_i$ is less troublesome. As such a surface we take the one with only one of principal curvatures nonvanishing. As such standard surface, we select the cylinder $\G$ with radius $R=-\kb^{-1}$, $\kb<0$ and we perform calculation in more detail than it was done in Sect.4. This enables us to determine the dependence of the effective symbol on the Lam\'e constants. (It does not matter that the cylinder is non-compact since the symbol $M_{\i}(\x)$ is a local quantity.)

 We consider such cylinder $\G$ as being described by the equation $x_1^2+(x_3+R)^2=R^2$; in a neighborhood of`the point $\xbc$ placed in  the origin,  $\xbc:\,x_1=x_2=x_3=0$, we can write
\begin{equation}\label{cylinder}
    x_3=F(x)= -R+\sqrt{R^2- x_1^2}=\frac12\kb x_1^2+O(x_1^4),\,x_1\to 0,\, x_2\in\R^1,
\end{equation}
with the normal vector directed  outside  the body, i.e., upward along the $x_3$ axis at the point $\xb^{\circ}=(0,0,0).$ The C-co-ordinates lines are, due to our construction in Section \ref{subsect coord}, directed along the curvature lines, i.e., along the orthogonal cross-section of the cylinder and along its axis. The frame vectors are directed along $x_1,x_2,x_3$ axes in $\R^3\subset \C^3$.
In the chosen co-ordinates, $\kb_1(\xbc)=\kb, \kb_2(\xbc)=0,$ and  derivatives of all entries in the direction $x_2$ vanish.

The first and second fundamental forms of the surface $G$ at the point $\xbc$ are  equal to
\begin{equation*}
    \Ib(\xbc,dx)=dx_1^2+dx_2^2; \, \Ib\Ib(\xbc,dx)=\kb dx_1^2.
\end{equation*}

The normal vector $\nup(\yb)=\nup(y_1,y_2,F(y_1))$ at the point $\yb=(y,F(y)) =(y_1,y_2,F(y_1,y_2))$ lying close to $\xbc$ equals
\begin{equation}\label{Normal}
\nup(\yb)=(-\kb y_1,0,\sqrt{1-\kb^2y_1^2})^{\top}=(-\kb y_1,0,1)^{\top}+O(y_1^2), y_1\to 0.
\end{equation}
Therefore, its derivatives are
\begin{equation*}
     \frac{\partial \nup(\yb)}{\partial y_1}=(-\kb,0,0)^{\top}+O(|y_1|); \, \frac{\partial \nup(\yb)}{\partial y_2}=0.
\end{equation*}

\subsection{The expansion of the kernel}


Recall that the kernel $\Kc(\xb,\yb)$ of the NP operator $\KF$ is given in \eqref{kernelNP}. We will express it in co-ordinates $x,y$ and find
the first two power terms in the expansion of $\Kc(y,y-x)$, as $y\to x$.

In  the C--co-ordinate system centered at $\xb=\xbc=0$,
we  calculate contributions to the singularities of the symbol of $\KF$ coming from the terms in \eqref{kernelNP} separately. First, we find the expansion of the distance $|\xb-\yb|=|\yb|$ and its powers. We have:
 $|\yb|^2=y_1^2+y_2^2 +F(y_1)^2$, therefore
\begin{equation*}
   |\yb|^{-2} =|y|^{-2}\left(1+\frac{F(y_1)^2}{|y|^2}\right)^{-1}=|y|^{-2}+O(1)).
\end{equation*}
Similarly,
\begin{equation*}\label{dist-3}
    |\yb|^{-3}=|y|^{-3}(1+O(|y|^2))=|y|^{-3}+O(|y|^{-1}), \, y\to 0.
\end{equation*}
Thus, when calculating the leading two terms in the expansion in \eqref{KrnelExpnsion}, we may replace $|\yb|$ by $|y|$.

Next we evaluate the expression $S_1(\yb)=-\sum_{l=1}^3{\n_l(\yb)y_l}$ entering in the second line \eqref{kernelNP}, for $\yb=(y,F(y_1))\in \G.$ Since $\nup(\yb)=(-\kb y_1,0,1)^{\top}+O(y_1^2)$, we have
\begin{equation*}
    S_1=-(-\kb y_1,0,1)(y_1,y_2,\frac12 \kb y_1^2)^{\top} +O(|y|^3) =\frac12\kb y_1^2+O(|y|^3),\, y\to 0.
\end{equation*}
Further on, we calculate the expression $S_{2,p,q}(y)=(x_p-y_p)(x_q-y_q)$ in \eqref{kernelNP}, with  $x_p=x_q=0$.
On the matrix diagonal, i.e., for $p=q$,
\begin{equation*}
S_{2,1,1}=y_1^2;\,S_{2,2,2} =y_2^2;\,S_{2,3,3} =F(y_1)^2=O(y_1^4),
\end{equation*}
while off-diagonal, for $p\ne q$,
\begin{gather*}
S_{2,1,2}=S_{2,2,1}=y_1 y_2;\, S_{2,1,3}=S_{2,3,1}=F(y_1)y_1=O(|y|^3);\\\nonumber S_{2,3,2}=S_{2,2,3}=F(y_1)^2=O(|y|^4).
\end{gather*}
Therefore, the expression on the second line in \eqref{kernelNP} equals, for $\xb=0$,
\begin{equation*}
[\Kc^{(2)}(y)]_{p,q}=\frac1{2\pi}\left[\km\d_{pq} +\frac{3\mm}{2} \frac{y_py_q}{|y|^2}\right]\frac{\kb y_1^2}{|y|^3}+O(1), p,q=1,2,
\end{equation*}

\begin{equation*}
    [\Kc^{(2)}(y)]_{p,q}=O(1),\,\mathrm{for}\, p=3,q=2\, \mathrm{or} \,p=2,q=3,
\end{equation*}
and, finally,
\begin{equation*}
    [\Kc^{(2)}(y)]_{3,3}=-\frac1{2\pi}\km \kb\frac{y_1^2}{|y|^3}+O(1)
\end{equation*}
(recall that $\km=\frac{\m}{2(\l+2\m)}$).

In particular, we can see that the expression on the second line in \eqref{kernelNP} has singularity of order $-1$ in $|x-y|$, and, therefore, contributes only to the subsymbol of the pseudodifferential operator $\KF$, but not to its principal symbol.

Next we calculate  the expansion  on the first line in \eqref{kernelNP} (recall, we set $\xb=\xbc=0$ here.) These terms may contribute to both principal and subsymbol of $\KF$. Note first  that the matrix defined by this line is antisymmetric, therefore, the diagonal terms, $p=q$, vanish.

For the off-diagonal terms, using \eqref{Normal}, we obtain
\begin{equation*}
   \n_1(\yb)y_2-\n_2(\yb)y_1=-\kb y_1y_2;
\end{equation*}

\begin{equation*}
\n_1(\yb)y_3-\n_3(\yb)y_1=-y_1+O(|y|^3);
\end{equation*}

\begin{equation*}
\n_2(\yb)y_3-\n_3(\yb)y_2=-y_2 +O(|y|^3).
\end{equation*}
Therefore, the expression on the first line in \eqref{kernelNP} equals
\begin{gather}\label{symbol complete1}
    \Kc^{(1)}(y)=\frac1{2\pi} \km |y|^{-3}\left(
                  \begin{array}{ccc}
                    0 & 0 & -y_1 \\
                    0 & 0 & -y_2 \\
                    y_1 & y_2 & 0\\
                  \end{array}
                \right)\\\nonumber +\frac1{2\pi}\km\kb|y|^{-3}\left(
                                                           \begin{array}{ccc}
                                                             0 & -y_1y_2 & 0 \\
                                                             y_1y_2 & 0 & 0 \\
                                                             0 & 0 & 0 \\
                                                           \end{array}
                                                         \right),
 \end{gather}
   and the one on the second line is
   \begin{gather}\label{symbol complete2}
    \Kc^{(2)}(y)  =  \frac{\kb}{2\pi}\left(
                       \begin{array}{ccc}
                         \frac{3\mm}{2}\frac{ y_1^4}{|y|^5}+\km\frac{y_1^2}{|y|^3} & \frac{3\mm}{2}\frac{y_1^3 y_2}{|y|^5} & 0 \\
                        \frac{3\mm}{2}\frac{y_1^3 y_2}{|y|^5}& \frac{3\mm}{2}\frac{ y_1^2y_2^2}{|y|^5}+\km\frac{y_1^2}{|y|^3} & 0 \\
                         0 & 0 & \km\frac{y_1^2}{|y|^3} \\
                       \end{array}
                     \right) \\\nonumber
                     =\frac1{2\pi} \kb\km\frac{y_1^2}{|y|^3}\Eb+\frac3{4\pi}\kb\mm y_1^2|y|^{-5}\left(
                                                                                 \begin{array}{ccc}
                                                                                   y_1^2 & y_1y_2 & 0 \\
                                                                                   y_1y_2 & y_2^2 & 0 \\
                                                                                   0 & 0 & 0 \\
                                                                                 \end{array}
                                                                               \right).
\end{gather}

The first term in  \eqref{symbol complete1}  corresponds to the principal part of the symbol of the NP operator only, while \eqref{symbol complete2} and the second term in \eqref{symbol complete1} contribute to the subsymbol.

\subsection{The expansion of the symbol}
Next we transform our formulas for the kernels \eqref{symbol complete1}, \eqref{symbol complete2} of the integral operators to the corresponding expressions for the components  of the symbol of the pseudodifferential operator $\KF$. Recall that this symbol is the Fourier transform of the kernel $\Kc$ of the integral operator in $x-y\in \R^2$ variable. Some of formulas we use can be found in standard tables of the Fourier transform of distributions, other ones need to be calculated by hand.

We start with recalling that the kernel $(2\pi|x-y|)^{-1}$ in $\R^2$ corresponds, by means of the Fourier transform $\Fs$ in $\R^2$, to the symbol $|\x|^{-1}$,
\begin{equation*}
    \Fs [(2\pi |y|)^{-1}]=|\x|^{-1}.
\end{equation*}

Further on, since $\frac{y_p}{|y|^{3}}=-\partial_p(|y|^{-1})$, $p=1,2,$
 \begin{equation*}
\Fs\left[(2\pi)^{-1}\frac{y_p}{|y|^{3}} \right]=i\x_p|\x|^{-1},
 \end{equation*}
in the sense of distributions.

 Next, we consider the kernel $|y|^{-5}$. Its Fourier transform in the sense of distributions equals $\frac{2\pi}{9}|\x|^3.$ By the properties of the Fourier transform, we have
 \begin{equation*}
    \Fs\left[\frac{y_1^3y_2}{2\pi|y|^5}\right]=\frac{1}{9} \partial_{\x_2}\partial_{\x_1}^3(|\x|^3)= \x_2^3\x_1|\x|^{-5}.
 \end{equation*}

In the same way,
\begin{equation*}
    \Fs\left[\frac{y_1^4}{2 \pi|y|^5}\right]=\frac{1}{9}\partial_{\x_1}^4(|\x|^3)= \x_2^4|\x|^{-5}.
\end{equation*}

Finally,
\begin{equation}\label{deriv4}
   \Fs\left[ \frac{y_1^2}{2\pi|y|^3}\right]= \x_2^2|\x|^{-3},
\end{equation}
\begin{equation}\label{deriv5}
   \Fs\left[ \frac{y_1y_2}{2\pi|y|^3}\right]= \x_1\x_2|\x|^{-3},
\end{equation}
and
\begin{equation}\label{deriv6}
   \Fs\left[ \frac{y_2^2}{2\pi|y|^3}\right]= \x_1^2|\x|^{-3}.
\end{equation}

As a result, the symbol of the NP operator $\KF$ on the surface $\G$ equals, in C-co-ordinates centered at the point $\xbc=(0,0,0)$,
\begin{gather}\label{symbol KF}
    \kF(\xbc,\x)={i\mathbbm{k}  }\left(
                   \begin{array}{ccc}
                     0 & 0 & -\x_1|\x|^{-1} \\
                     0 & 0 & -\x_2|\x|^{-1} \\
                     \x_1|\x|^{-1} & \x_2|\x|^{-1} & 0\\
                   \end{array}
                 \right)+ \\\nonumber \frac12 \kb\km{\x_2^2}{|\x|^{-3}}\Eb
                 + \frac32\kb\mm|\x|^{-5}\x_2^2\left(
                           \begin{array}{ccc}
                             \x_2^2& \x_1\x_2 & 0 \\
                              \x_1\x_2 & \x_1^2 & 0 \\
                             0 & 0 &  \\
                           \end{array}
                         \right)+ \\\nonumber \kb\mm|\x|^{-3}\left(
                                                   \begin{array}{ccc}
                                                     0 & \x_1\x_2 & 0 \\
                                                     -\x_1\x_2 & 0 & 0 \\
                                                     0 & 0 & 0 \\
                                                   \end{array}
                                                 \right)
                        + O(1).
\end{gather}

So, we have calculated the leading two terms in the expansion of the symbol of $\KF$ for the case of a cylinder.  In the expression  \eqref{symbol KF}, the  first term is order $-0$ homogeneous and represents the leading symbol $\kF_0,$ as we already know, while the remaining  terms, the ones on the second and third lines,  represent the subprincipal symbol $\kF_{-1}$. We note that both $\kF_0$ and $\kF_{-1}$ are Hermitian. Moreover, $\kF_0$ does not depend on the geometry of the surface $\G$, while $\kF_{-1}$ depends linearly on the curvature $\kb$, as we have already found from less detailed considerations in Sect.4. Therefore, for a general, non-cylindrical  surface, by the curvature linearity established in Sect.4, a similar term containing the second principal curvature should be added to the subsymbol. Additionally, we pinpoint that the subsymbol terms in \eqref{symbol KF} depend linearly on the material characteristics $\km=\frac{\m}{2(2\m+\l)}$ and $\mm=\frac{\l+\m}{2(2\m+\l)}=\frac12-\km.$
These properties of the symbol of $\KF$ will be discussed further
on.

\subsection{The gradient of the principal symbol}\label{Section.Gradient}
Next, we need to find an explicit expression of one more object, namely the gradient of the principal symbol $\kF_0(x,\x)$ that enters in the expression for $\mF_{\i}$. We note, for a further reference, that the principal symbol $\kF_0$ depends linearly  on $\km$ but does not depend on $\mm$. The same property is valid for all derivatives of $\kF_0.$

Formula \eqref{symbol KF} gives us the representation of this symbol at the point $\xbc=0$ in the co-ordinates and the frame associated with this point. What we need now is to calculate this symbol at a different point $\xt=(x,F(x))\ne \xbc $ and then find derivatives $\partial_{x_\a}\kF_0(x,\x)$ and $\partial_{\x_\a}\kF_0(x,\x)$ for $x=0$ by means of making $\xt$ approaching $\xbc$. Again, a condense description of this calculatoin is presented in Sect.4.   This calculation is needed for $\a=1$ only, since, on the cylinder, all derivatives in $x_2$ variable vanish. We find the derivative in $\x_1$ first. By a direct calculation, we obtain from \eqref{symbol KF}:

\begin{equation*}
    \partial_{\x_1}\kF_0(x,\x)= \frac{ i\mathbbm{k}}{|\x|^3}\left(
                                  \begin{array}{ccc}
                                    0 & 0 & \x_2^2 \\
                                    0 & 0 & \x_1\x_2 \\
                                    -\x_2^2& -\x_1\x_2 &0  \\
                                  \end{array}
                                \right).
\end{equation*}
Next we find the derivative $\partial_{x_1}\kF_0(x,\x)$. For a given $x_1$, we take a point $\xt=(x_1,0,F(x_1))\in\G$ and consider the C--co-ordinates system $(y_1,y_2,y_3)$ centered at $\xt$ and the corresponding frame in $\R^3$. These co-ordinates are rotated compared with the system centered at $\xbc$, the rotation realized by the matrix
\begin{equation}\label{U}
    U(\xt)=\left(
       \begin{array}{ccc}
         \cos(\theta) & 0 & \sin(\theta) \\
         0 & 1 & 0 \\
         -\sin(\theta) & 0 & \cos(\theta) \\
       \end{array}
     \right),
\end{equation}
where the angle $\theta$ equals
\begin{equation}\label{angle}
\theta=\arcsin{F'(x_1)}=\arcsin(\frac{x_1}{\kb}).
  \end{equation}
  This is a rotation around the $y_2$ axis, directed, recall, along the directrix of the cylinder $\G$, i.e., parallel to the  $x_2$ axis.
The frame in $\R^3$ at the point $\xt$ has directions along the axes $y_1,y_2,y_3;$ the covariables $\y_j$ are directed along the corresponding  $y_j$ axes.

Further on, since we need to trace the dependence of the symbols of our operator on the co-ordinate system, we will mark it by the superscript: thus, $\kF_0^\circ$ denotes the expression of the symbol in the C--co-ordinate system and frame centered at $\xbc$ etc.

Now, the principal symbol of the operator $\KF$ at the point $\xt$, calculated in $(\yb,\y)$-  C--co-ordinate system, is given by the same expression as in \eqref{symbol KF}, just with $\x$ replaced by $\y:$
\begin{equation}\label{symbZ}
    \kF_0^{\bullet}(\xt,\y)={ i\mathbbm{k}}{|\y|^{-1}}\left(
                           \begin{array}{ccc}
                             0 & 0 & -\y_1 \\
                             0 & 0 & -\y_2 \\
                             \y_1 & \y_2 & 0 \\
                           \end{array}
                         \right).
\end{equation}
Now we apply the rule of the variables change in the principal symbol of pseudodifferential operators;
 \begin{equation*}\label{change}
    \kF_0^{\xb}(\xt,\x)=U(x_1)^{-1}\kF^{\bullet}_0({\xt,U(\yb)^{*}\x})U(x_1),
\end{equation*}
where $U=U(x_1)$ is the linear transformation \eqref{U}. Note, that in addition to the standard formula for the  change of variables in pseudodifferential operators, which  is reflected by the presence of $U $ in the argument in $\kF$, formula \eqref{symbZ} takes into account the circumstance that the symbols at the points $\xt$ and $\xbc$ are represented in different frames, related, again, by means of the same matrix $U(x_1)$. This gives us
\begin{gather}\label{Trans2}
   \kF_0^{\xt}(\xbc,\x)=\frac{ i\mathbbm{k}}{|\x|}U^{-1}\left(
                          \begin{array}{ccc}
                            \x_1\sin\theta & 0 & -\x_1\cos\theta \\
                            0 & 0 & -\x_2 \\
                            \x_1\cos\theta & \x_2 & \x_1\sin\theta \\
                          \end{array}
                        \right)U=\\\nonumber\frac{ i\mathbbm{k}}{|\x|}\left(
                          \begin{array}{ccc}
                            \x_1\sin\theta & 0 & -\x_1\cos\theta \\
                            0 & 0 & -\x_2 \\
                            \x_1\cos\theta & \x_2 & \x_1\sin\theta \\
                          \end{array}
                        \right).
\end{gather}
The last equality uses the fact that, in our case, the matrix $U$ commutes with the matrix $\kF_0$.
Thus, we have  obtained the expression for the principal symbol of $\KF$ written in one and the same co-ordinate system and the same frame. Now we  can differentiate the expression \eqref{Trans2} in $x_1$ variable, taking into account \eqref{angle}:
\begin{equation}\label{deriv.x}
    \frac{\partial}{\partial x_1}\kF^{\xbc}_0(x_1,0,F(x_1),\x)\mid_{x_1=0}=\frac{ i\km\kb}{|\x|}\left(
                          \begin{array}{ccc}
                            \x_1 & 0 & 0\\
                            0 & 0 & 0 \\
                            0 & 0& \x_1\ \\
                          \end{array}
                        \right).
\end{equation}
Finally, we collect all terms we calculated in the symbols; it is more graphical to represent them  in homogeneous variables $\f_\b=\x_\b/|\x|, \, \b=1,2$, so that $\f_1^2+\f_2^2=1$. In this way we have
\begin{gather*}
    \kF_0(\x)=i\km\left(
                             \begin{array}{ccc}
                               0 & 0 & -\f_1 \\
                               0 & 0 & -\f_2 \\
                               \f_1 & \f_2 & 0 \\
                             \end{array}
                           \right);\,
                           \kF_{-1}(\x)=|\x|^{-1}\kb(\km \uF(\x)+\mm\vF(\x)),\\ \nonumber
\uF(\x)=-\f_2^2\Eb;\, \vF (\x)= -\frac12\f_2^2\left(
                            \begin{array}{ccc}
                           \f_2^2 & \f_1\f_2 & 0 \\
                            \f_1 \f_2& \f_1^2 & 0\\
                              0 & 0 & 0 \\
                              \end{array}
                              \right) -\frac32\left(
                                                        \begin{array}{ccc}
                                                          0 & \f_1\f_2 & 0 \\
                                                          -\f_1 \f_2& 0 & 0 \\
                                                          0 & 0 & 0 \\
                                                        \end{array}
                                                      \right)
                                       ;\\\nonumber
               \partial_{\x_1}\kF_0(x,\x) =i\km |\x|^{-1}\left(
                                               \begin{array}{ccc}
                                                 0 & 0 & -\f_2^2 \\
                                                 0 & 0 & -\f_1\f_2 \\
                                                 \f_2^2 & \f_1\f_2 & 0 \\
                                               \end{array}
                                             \right);\\\nonumber
             \partial_{x_1}\kF_0(x,\x)=i\km\kb|\x|^{-1}\left(
                                         \begin{array}{ccc}
                                           \f_1 & 0 & 0 \\
                                           0 & 0 & 0 \\
                                           0 & 0 & \f_1 \\
                                         \end{array}
                                       \right).                                            \end{gather*}

\subsection{The symbol $\mF_\i$}

By the reasons discussed earlier, the explicit expression of the symbol of the order $-1$ pseudodifferential operator $\MF_\i$ is rather wild. Even in the case of a cylinder,  the task of calculating the eigenvalues of the symbol remains quite irrational.  This calculation would involves 15 products of matrices, leading to a incomprehensible expression. It would become even more obscure and senseless after adding 10 more terms present if \emph{both } principal curvatures are nonzero.To illustrate the above, we write down all terms in the terms for the cylinder case in the symbol $\mF_{-1}$.

To shorten the notations, we write $\kF$ instead of $\kF_0,$  $\jF$ instead of $\kF_{-1}$, $\gF$ instead of $\partial_{\x_1}\kF$, and $\hF$ instead of $\partial_{x_1}\kF$, and also omit the subscript $\a$ in \eqref{type2}.

So, for calculating $\mF_{-1}$ (recall, this is the effective symbol for the point $\o_{-1}=-\kn$), we compose terms of type 1, see \eqref{type1} and find their sum:
\begin{gather}\label{type1.expl}
\Fb:=  F_1+F_2+F_3+F_4+F_5=\jF\kF^2(\kF-\km)^2 +(\kF+\km)\jF\kF (\kF-\km)^2\\\nonumber
  + (\kF+\km)\kF \jF(\kF-\km)^2+(\kF+\km)\kF^2 \jF(\kF-\km)+(\kF+\km)\kF^2 \kF-\km)\jF.
\end{gather}
Then we determine the terms of type 2, see \eqref{type2}, and write down  their sum:
\begin{gather}\label{type2.exp}
  \Gb:=\\\nonumber -\imath [G_{1,2}+ G_{1,3}+G_{1,4}+G_{1,5}+G_{2,3}+G_{2,4}+G_{2,5}+G_{3,4}+G_{3,5}+G_{4,5}]=\\\nonumber
  -\imath[\gF\hF\kF(\kF-\km)^2+\gF \kF\hF(\kF-\km)^2+\gF \kF(\kF-\km) \hF(\kF-\km)+\gF \kF(\kF-\km)^2 \hF+\\\nonumber (\kF+\km)\gF\hF(\kF-\km)^2+
  (\kF+\km)\gF\kF\hF(\kF-\km)+(\kF+\km)\gF\kF(\kF-\km)\hF+ \\\nonumber (\kF+\km)\kF\gF\hF(\kF-\km)+(\kF+\km)\kF\gF(\kF-\km)\hF+(\kF+\km)\kF^2\gF\hF].
\end{gather}

Finally, $\mF_{-1}=(\pb'(\o_\i))^{-1}(\Fb+\Gb),$ it is the sum of 15 terms, for a cylinder with radius $\kb,$ and the coefficient $M_{-1}(\xi)$ is equal to  $M_{-1}(\xi)=\kb^{-1}\mF_{-1}(\x)$. This coefficient is universal, and for  a general surface with principal curvatures $\kb_1(\xb), \kb_2(x)$, the symbol is calculated using the linearity in the principal curvatures, see \eqref{perm.superfin}. The integrand in the asymptotic formulas is, up to some constant,

For calculating $\mF_\i$ for $\i=0$ and for $\i=1$, the formulas are analogous to \eqref{type1.expl}, \eqref{type2.exp}, but with  different combinations of $\kF-\o_\i,$ in accordance with \eqref{type1}, \eqref{type2}.

\subsection{Dependence on the material}
The explicit  expression for $\mF(x,\x)$ is extremely cumbersome; if typed, it would fill several unreadable pages.  We, however, have already determined its  dependence on the geometry of the body $\Dc. $  What we can do now is to describe its dependence  on the material of the body, namely on the Lam\'e constants.

When analysing the expression \eqref{type1.expl},  we see that  in each of 5 summands  (the last one vanishes, but this does not matter) there are 4 factors containing the principal symbol $\kF=\kF_0(\x)$,  each of them contains the factor $\kn$. Additionally, one of factors is the symbol $\jF,$ \emph{linear in $\km $ and $\mm$.}  The above 4 entries of $\km$ are cancelled by the factor $(\pb_{\i}'(\o_\i))^{-1}$ in \eqref{complete symbol} which contains $\km^{-4}$.  What remains, is that all   summands in \eqref{type1.expl} produce a contribution to $\mF_\i$, being linear form   linearly of $\km$ and $\mm$ with universal coefficients which are some matrices depending only on $\x$ (and, of course, on $\i$).
Let us pass to the terms of type 2 in \eqref{type2.exp}.  Each summand in $\Gb$ is a product of five terms, of which 3 contain the principal symbol $\kF=\kF_0(x,\x)$, the other two are the gradient in $\x$ and the gradient in $x$ variables of $\kF.$ As follows from our calculations above, each of these factors in  \eqref{type2} contain the factor $\km$, altogether 5 of them. Four of them are canceled, again, by the coefficient $(\pb_{\i}'(\o_i))^{-1}$, therefore, again, a linear dependence of the parameter $\km$ remains in $\mF_{\i}$ but  not a dependence on $\mm$ any more. We sum this reasoning by stating that the contribution of the term $\Gb$ in  \eqref{complete symbol} to $\mF_{\i}$ is a linear function of quantities $\km=\frac{\m}{2(\l+2\m)}$  with a universal coefficient.

Altogether, we obtain that the matrix $M_{\i}(\x)$ is a linear form of $\km$ and $\mm=\frac12-\km$ with universal coefficients.
Therefore, \eqref{Formula3} involves the eigenvalues of a matrix depending linearly on the material parameters $\km$ and $\mm$  and linearly on the principal curvatures. We arrive at the following representation for the effective symbol.
\begin{theorem}\label{Thm.dependence}
  The effective symbol $\mF_\i$ satisfies
\begin{equation}\label{subsymbol}
\mF_{\i}(x,\x)=\kb_1(x)(\kn +\mm Y_{\i}(\x) )+\kb_2(x) V^{-1}(\kn X_{\i}(\hat{\x})+\mm Y_{\i}(\hat{\x}) )V^{-1}.
\end{equation}
with  universal matrix functions $X_{\i}(\x),Y_{\i}(\x),$ order $-1$ positively  homogeneous in $\x\in\R^2.$
\end{theorem}
We obtain more knowledge  on the matrices $\kn X_{i}(\x)$, $Y_{\i}(\x)$ after some further analysis of the spherical case.

\subsection{Returning to the sphere} Having the expression \eqref{subsymbol} for the symbol $\mF_{\i}$, we return to the result \eqref{two-sided} for the two-sided asymptotics of eigenvalues,
We obtain the following expressions for the coefficients $\Ac_\i$ and $\Bc_\i$ in \eqref{two-sided}.

\begin{gather}\label{Ac.new}
 \Ac_\i=
 2\pi^{-1} \left[\int_{S^1}\kn^2 \tr (X_{\i}(\x))^2\pmb{\pmb{\o}} +2\int_{S^1}\kn\mm  \tr(X_{\i}(\x)Y_{\i}(\x))\pmb{\pmb{\o}}+\mm^2\int_{S^1}\tr Y^{2}_{\i}(\x)\pmb{\pmb{\o}}\right]\\\nonumber =\pmb{\Omega}_{\i,W}(\kn,\mm),
\end{gather}

and a similar expression for $\Bc_{\i}$,

\begin{gather}\label{Bc.new}
  \Bc_{\i} = 2
 (\int_{S^1} \tr\left[(\kn X_\i(\x)+\mm Y_\i(\x))V^{-1}(\kn X_\i(\x)+\mm Y_\i(\x))V\right]\pmb{\pmb{\o}} - \Ac_{\i} \\\nonumber
  = \pmb{\Omega}_{\i,\chi}(\kn,\mm).
\end{gather}
In these formulas  $\pmb{\Omega}_{\i,W}(\kn,\mm)$, $\pmb{\Omega}_{\i,\chi}(\kn,\mm)$ are quadratic forms of the quantities $\kn,\mm$, with universal coefficients depending only on $\i.$

In particular, these coefficients have the same form for the body being the unit ball; the Euler characteristic and the Willmore energy for the sphere equal, correspondingly, $\chi(S^2)=2,$ $W(S^2)=4\pi$. We compare \eqref{Ac.new}, \eqref{Bc.new}  with the asymptotic formulas \eqref{ball.counting}. Due to the fact that both formulas \eqref{ball.counting} and \eqref{two-sided} with coefficients \eqref{Ac.new},
\eqref{Bc.new} must give the same result for all combinations of the Lam\'e constants $\l,\m$,  we derive from this comparison the following. In particular, the coefficient in the eigenvalue asymptotics for a sphere, see \eqref{ball.counting}, for $\i=0$ does not depend on the Lam\'e constants, therefore, all their entries for $\i=0$ cancel due to $\kn+\mm=\frac12$.

\begin{theorem}\label{Th.Coeff}The quadratic form $\Upsilon_{\i}(\kn,\mm)=2\pi^{-1} \Ac_{\i}(\kn,\mm)+2{\Bc_{\i}}$ possesses the following properties
\begin{itemize}
               \item  For $\i=\pm1,$ the form $\Upsilon_{\i}(\kn,\mm)$ equals $\g_{\i}\kn^2,$
               \item  For $\i=0$, $\Upsilon_{\i}(\kn,\mm)= \g_0,$
             \end{itemize}
  Here $\g_\i$ are  absolute numeric  coefficients obtained by integrating expressions containing the matrix-functions $X_\i,Y_\i,$ see \eqref{Ac.new}, \eqref{Bc.new}.
\end{theorem}

We can also note that for the sphere, there are no eigenvalues approaching $\o_{\i}$ from below. This means that the symbol \eqref{subsymbol} is non-negative for all $\i,\km,\x.$
\appendix
\section{Some remarks on the nonhomogeneous material} In the paper \cite{MR3D} an approach was made to consider a mathematical model of a \emph{nonhomogeneous} elastic body (with the Lam\'e) parameters $\l(\xb),\m(\xb)$ being smooth, possibly nonconstant, functions of the point $\xb\in\Dc.$ A possible version of the NP operator was constructed. It was found that the essential spectrum of this operator $\KF$ may fill nontrivial intervals
of the real line, namely, the set of values of the, now functions, $\km(\xb)=\frac{\m(\xb)}{2(2\m(\xb)+\l(\xb))}$ and $-\km(\xb)$ for $b\x\in\G$, as well as the  point $0$.  This fact was based upon the representation,  found in \cite{MR3D}, of $\KF$ as an order zero pseudodifferential operator with principal symbol
\begin{equation}\label{NPPrincVar}
    \kF_0({x},\x)=\frac{i \mathbbm{k}(x)}{|\x|}\begin{pmatrix}
                    0 & 0 & -\x_1 \\
                    0 & 0 & -\x_2 \\
                    \x_1 & \x_2 & 0 \\
                  \end{pmatrix},
 \end{equation}
 $(x,\x)\in T^*\G,$
in the same local co-ordinate system and local frame, as we use here for a homogeneous material. A  question arises about the eigenvalues converging to the tips of the essential spectrum, namely, to the boundary points of the above intervals.  It is natural to expect that the character of this convergence should depend on the structure of these boundary points.

The starting point here  is the case of a nondegenerate extremal point of $\km(x)$ at the boundary point. If $\km(x),$ say, has a nondegenerate maximal point at $\xb^\circ\in\G$ then the corresponding eigenvalue $\s(x,\x)$ of the principal symbol $\kF_0(x,\x)$ is equal to $\km(\xb^\circ)$ and therefore does not depend on $\x.$ It has its extremal value, moreover a nondegenerate one, for all $\x.$

Another   special (and quite convenient) property concerns the subsymbol of $\KF.$ As the reasoning in the present paper shows, the subsymbol of $\KF$ is constructed using the Taylor expansion of the integral kernel of the operator $\KF$. This reasoning, based on the explicit formulas in \cite{MR3D}, can be performed analogously to the one in Sect. 5 in our present paper. Unlike our present case, for a nonhomogeneous material, this expansion would involve not only geometrical characteristics (principal curvatures) of $\G,$ but  also the derivatives of the function $\km(x)$ along directions on $\G.$ Fortunately, in the extremal point of $\km(x),$ the first order derivatives of this function vanish, therefore, the expression for the subsymbol of $\KF$ at this point turns out to be the same as for the homogeneous case, thus given by \eqref{symbol KF}. A construction has been performed in  \cite{RozNP1}, showing that under the above conditions, the knowledge of the second order jet of the eigenvectors of $\kF_0(x,\x)$  at the point $\xb^\circ,$ together of the subsymbol $\kF_{-1}$ at this point suffices to find the asymptotics of eigenvalues converging to
$\km(\xb^\circ).$

\section*{Acknowledgments} the Author is grateful to Y.Miyanishi for introducing him to the NP problematic and for useful discussions. .

\section*{Declarations}

\begin{itemize}
\item Funding.  The Author was supported  by the grant from the Russian Fund of Basic Research 20-01-00451, (Sections 1-4) and the grant from the Russian Science Foundation, Project 20-11-20032 (Sections 5,6)
\item Conflict of interest/Competing interests : the Author declares no conflict of interest.
\item Consent to participate: Not applicable
\item Consent for publication: Not applicable
\item Availability of data and materials: Not applicable
\item Code availability: Not applicable
\item Authors' contributions: Not applicable
\end{itemize}

\end{document}